\documentclass[11pt]{article}
\usepackage{amsmath,amssymb}
\usepackage{geometry}
\usepackage{color}
\usepackage{layout}

\usepackage{hyperref}
\newtheorem{theorem}{Theorem}[section]
\newtheorem{lemma}[theorem]{Lemma}

\newtheorem{corollary}[theorem]{Corollary}
\newtheorem{proposition}[theorem]{Proposition}

\newtheorem{remark}[theorem]{Remark}

\newtheorem{notation}[theorem]{Notation}

\newtheorem{question}[theorem]{Question}

\newcommand{\aut}{{\rm Aut\/}}

\newcommand{\PSL}{{\rm PSL\/}}

\newenvironment{proof}    {\par\noindent{\textit{Proof}\/: }\nopagebreak\normalsize}    {\hfill\linebreak[0]\hspace*{\fill} {\rule{1.5mm}{2mm}} \\[1pt]}

\begin{document}
\title{Fenchel's conjecture on NEC groups}
\author{E. Bujalance\thanks{Partially supported by PID2023-152822NB-I00}, F.J. Cirre\thanks{Partially supported by PID2023-152822NB-I00}, M.D.E. Conder\thanks{Partially supported by NZ Marsden Fund UOA2320}, A.F. Costa\thanks{Partially supported by PID2023-152822NB-I00}}
\maketitle
\begin{abstract}
A classical discovery known as Fenchel's conjecture and proved in the 1950s, 
shows that every co-compact Fuchsian group $F$ has a normal subgroup of finite index 
isomorphic to the fundamental group of a compact unbordered orientable surface, 
or in algebraic terms, 
that $F$ has a normal subgroup of finite index that contains no element of finite order other than the identity.
In this paper we initiate and make progress on an extension of Fenchel's conjecture by considering 
the following question: 
Does every planar non-Euclidean crystallographic group $\Gamma$ containing transformations that 
reverse orientation have a normal subgroup of finite index isomorphic to the fundamental group of a
compact unbordered non-orientable surface? 
We answer this question in the affirmative in the case where the orbit space of $\Gamma$ is a nonorientable 
surface, and also in the case where this orbit space is a bordered orientable surface of positive genus. 
In the case where the genus of the quotient is $0$, we have an affirmative answer in many subcases, 
but the question is still open for others. 
\end{abstract}

\section{Introduction}
\label{introduction}
Let $\Gamma$ be a co-compact crystallographic group $\Gamma$ of isometries of the hyperbolic
plane, also known as a non-Euclidean crystallographic group (or NEC group, for short). 
in this paper we investigate torsion-free normal subgroups of finite index in such a group $\Gamma $. 
Algebraically there are two different kinds of  torsion-free normal subgroups:  those uniformising 
an orientable Klein surface, and those uniformising a non-orientable surface.
We study the question of whether or not there exist such torsion-free normal subgroups of each kind.

Fuchsian groups can be considered as a particular type of NEC groups -- see \cite{wilkie}, in which 
canonical presentations of NEC groups were obtained in an analogous way to those for Fuchsian groups.
Hence the above question can be viewed as an extension of Fenchel's conjecture, 
which asserted that every co-compact Fuchsian group $F$ has a normal subgroup of finite index 
isomorphic to the fundamental group of a compact unbordered orientable surface,

Fenchel's conjecture was proved in complete generality in a sequence of papers
by J. Nielsen \cite{nielsen}, S. Bundgaard and J. Nielsen \cite{bundgaard-nielsen}, and R.H. Fox \cite{fox}, 
with a correction by T.C. Chau \cite{chau}; see also \cite{mennicke} for another complete proof by Mennicke. 

The classical algebraic version of the resulting theorem may be stated as follows:

\begin{theorem}[\protect\cite{nielsen,bundgaard-nielsen,fox}]
\label{Fenchel's.conjecture.classic}
Let $F$ be the finitely-presented group generated by $2g+r$
elements $a_{1},b_{1},\ldots ,a_{g},b_{g},$ $x_{1},\ldots ,x_{r},$
subject to the defining relations $x_{1}^{m_{1}}=\dots = x_{r}^{m_{r}} = 1$  and
$[a_{1},b_{1}]\ldots [a_{g},b_{g}]\,x_{1}\ldots x_{r}=1.$ 
Then $F$ has a normal subgroup of finite index that contains no element of finite order other than the identity.
\end{theorem}

Note that each term $[a_{i},b_{i}]$ can stand for either $a_{i}b_{i}a_i^{-1}b_{i}^{-1}$ 
or $a_{i}^{-1}b_{i}^{-1}a_{i}b_{i}$ (by replacing $(a_i,b_i)$ by $(a_i^{-1},b_i^{-1})$ if necessary),  
but we will assume that a commutator $[u,v]$ is 
defined as $u^{-1}v^{-1}uv$ in the remainder of this paper. 

\smallskip

Note also that Theorem~\ref{Fenchel's.conjecture.classic} applies to groups that are isomorphic 
to spherical and Euclidean plane crystallographic groups, for which the conclusion is easy or well known.
Also by the Uniformisation Theorem, it has important consequences for the
study of groups of automorphisms of surfaces with geometrical structure.
In particular, every co-compact Fuchsian group is isomorphic to such a  finitely-presented group, 
and as a consequence, we have the following:

\begin{theorem}\label{Fenchel's.conjecture.fuchsian} 
Let $\Gamma $ be a co-compact Fuchsian group. 
Then $\Gamma $ has a normal subgroup $\Lambda$ of finite index
that contains no element of finite order other than the identity, and hence is isomorphic 
to the fundamental group of a compact surface (also known as a surface Fuchsian group). 
Equivalently, the given group $\Gamma $ occurs as an overgroup of a surface Fuchsian group 
$\Lambda$ that uniformises a compact Riemann surface $\mathcal{H}/\Lambda$, 
and this surface admits a finite group of automorphisms isomorphic to $\Gamma /\Lambda$.
\end{theorem}

Many issues related to Fenchel's conjecture have been investigated by other authors.
For example, Edmonds, Ewing and Kulkarni used geometrical methods 
in \cite{edmonds-ewing-kulkarni:bulletin,edmonds-ewing-kulkarni:inventiones} to determine the 
possible finite indices of torsion-free subgroups of a Fuchsian group, with respect to its torsion.
Some of their findings were also proved by Burns and Solitar in \cite{burns-solitar}, 
using purely algebraic methods.
The same question about indices of torsion-free subgroups in the case of NEC groups was solved using 
algebraic and geometric arguments, by Izquierdo in \cite{izquierdo.thesis,izquierdo.procLondonMathSoc}. 
In the latter two articles, the author revealed some of the special difficulties in working with NEC groups.
Here we also point out that all the above authors left aside the matter of normality of the subgroups, 
which appears to require deeper consideration. 

\medskip

One interesting interpretation of Theorem \ref{Fenchel's.conjecture.fuchsian} in the context of orbifolds 
is the following (considered also in \cite{izquierdo.thesis,izquierdo.procLondonMathSoc}):  
Every 2-orbifold with hyperbolic structure, whose underlying topological space is a compact unbordered
orientable surface, admits a finite regular orbifold covering that is a 2-manifold. 
Note that the orbifold structure refers to the finiteness of \textit{local} isotropy groups, 
while Theorem \ref{Fenchel's.conjecture.fuchsian} gives a global information -- namely that if 
an orientable 2-orbifold is compact, then such an orbifold is the space of orbits of the (global) action 
of a finite group on some surface.

\medskip

The authors of  \cite{edmonds-ewing-kulkarni:annals} determined conditions for the existence of 
regular tessellations of type $\{p,q\}$ on unbordered surfaces. They studied the cases of orientable 
and non-orientable surfaces separately, since the algebraic structure of discrete groups with 
orientation-reversing elements requires somewhat different techniques to those consisting of orientation-preserving elements.
Note here that the concept of a regular tessellation considered in \cite{edmonds-ewing-kulkarni:annals} 
does not assume or imply the existence of a group of automorphisms acting transitively on faces and 
ordered edges of the tessellation, so it does not apply directly to the main topic of our current investigation.

\medskip

In this paper, we consider the generalisation of Fenchel's conjecture to the setting of \textit{proper NEC groups}, 
which are co-compact non-Euclidean crystallographic groups that are not Fuchsian groups. 
Such a generalisation has consequences for co-compact Klein surfaces, because NEC groups play the 
same role for Klein surfaces as Fuchsian groups play for Riemann surfaces.

\medskip

A first generalisation would be to ask whether every proper NEC group
has a torsion-free normal subgroup of finite index, without any other
condition(s) on the subgroup. It is not hard to see that the answer is
affirmative, as in Theorem \ref{orientable.surface.subgroup} below. This theorem
shows the existence of a surface \emph{Fuchsian} normal subgroup of finite
index, that is, a finite-index subgroup consisting of elements which preserve the
orientation of the hyperbolic plane. As a consequence, every NEC group
determines a finite group action on some compact Riemann surface, or in terms 
of orbifolds, every compact hyperbolic 2-orbifold is the orbit space of the
action of some finite group of isometries of an orientable hyperbolic surface.

The role played by torsion-free Fuchsian groups in the wider setting of NEC groups 
is played by \emph{surface NEC groups}, so called because they
uniformise hyperbolic compact Klein surfaces.
In contrast to the Fuchsian situation, however, a surface NEC group may contain torsion elements, 
which are necessarily hyperbolic reflections, of order $2$.

Indeed a first observation that can be made about the possible extension of Fenchel's conjecture 
to other types of NEC\ groups is that not every NEC group containing hyperbolic reflections 
has a bordered bordered surface NEC group -- that is, a normal \emph{surface} subgroup 
of finite index containing hyperbolic reflections; see  Theorem \ref{bordered.surface.subgroup} below. 
This theorem determines the NEC groups which occur as the orbifold fundamental group of a 
quotient of a bordered Klein surface by the action of some group of automorphisms.

Much more difficult is the extension of Fenchel's conjecture to NEC\ groups
that uniformise unbordered non-orientable Klein surfaces. Note that the algebraic
structure of non-orientable NEC\ groups is essentially different from the
structure of orientable ones, as will be seen in Section \ref{preliminaries}. 
A more specific question we will focus on is this: 
\textquotedblleft Does every proper NEC group have  a torsion-free normal subgroup (surface NEC group) 
of finite index containing elements that reverse orientation?\textquotedblright 

If the latter question has a positive answer, then a consequence would be that every NEC group determines a finite group 
action on some unbordered non-orientable Klein surface -- or equivalently, in terms of orbifolds, that 
every compact hyperbolic 2-orbifold whose underlying topological space is a compact non-orientable 
or bordered surface is the orbit space of the action of some finite group of isometries of a non-orientable hyperbolic unbordered compact surface.

In this paper we obtain a positive answer to the above question in many cases, by proving the following:

\begin{theorem}
Let $\Gamma $ be a proper NEC group such that 
$\mathcal{H}/\Gamma $ is not homeomorphic to a  bordered surface of topological genus zero.
Then $\Gamma $ has a torsion-free normal subgroup of finite index that contains orientation-reversing elements.
\end{theorem}

We also obtain a positive answer in other cases, such as when $\mathcal{H}/\Gamma $ is bordered and has genus zero.
For example, if $\mathcal{H}/\Gamma $ is an orbifold of genus zero and the numbers $k_0$ and $k_3$ 
of its boundary components with respectively no corner points or at least three corner points satisfy $k_{0}+k_{3}\geq 2$, 
then the orbifold $\mathcal{H}/\Gamma $ has a covering that is a non-orientable surface (see Theorem~\ref{theorem.non-orientable.g=0.k_0+k_3>1}). 

As an indication of the level of difficulty of finding an answer in some of the remaining cases, 
observations near the end of this paper will relate the question to a recent discovery about 
non-orientable regular maps with given hyperbolic type, and to the existence 
of non-orientable abstract regular polytopes with given Schl\"afli type.

\section{Preliminaries}
\label{preliminaries}

In this section we briefly describe some of the main facts about NEC groups that will be used in the paper.
For a general and more detailed account of this topic, we refer the reader to Section 0.2 in the book 
\cite{begg}. 

\medskip

First, an NEC group $\Gamma$ is a discrete subgroup of the group Isom$%
^{\pm}(\mathcal{H})$ of isometries of
the hyperbolic plane $\mathcal{H}$, preserving or reversing orientation, 
such that the orbit space $\mathcal{H}%
/\Gamma$ is compact. Any such group admits a presentation with up to four kinds of 
generators: 

\begin{itemize}
\item[$\bullet$] elliptic elements $\,x_i,\,$ for $1 \le i \le r$,

\item[$\bullet$] reflections $\,c_{i0}, \ldots , c_{is_i},\,$ for $1 \le i \le k$, 

\item[$\bullet$] orientation-preserving elements $\,e_i,\,$ for $1 \le i \le k$, \ and 

\item[$\bullet$] either hyperbolic elements $\,a_i$ and $b_i,\,$ 
or glide reflections $\,d_i,\,$ for $1 \le i \le g\,$;    
\end{itemize}

\noindent 
and these are subject to the defining relations below:

\begin{itemize}
\item[$\bullet$] ${x_i}^{\!m_i} = 1\,$ for $1 \leq i \leq r$, 

\item[$\bullet$] ${c_{ij-1}}^{\!2}={c_{ij}}^{\!2}=
(c_{ij-1}c_{ij})^{n_{ij}}=1\,$ for $1\leq j\leq s_i,$ for $1\leq i\leq k$,

\item[$\bullet$] $c_{is_i}=e_ic_{i0}{e_i}^{\!-1}\,$ for $1\leq i\leq k$, \ and 

\item[$\bullet$] $[a_1,b_1] \ldots [a_g,b_g]\, x_1 \ldots x_r \, e_1 \ldots e_k=1$ 
\ or  \ 
${d_1}^{\!2} \ldots {d_g}^{\!2\,}x_1 \ldots x_r \, e_1 \ldots e_k=1$. 
\end{itemize}

A set of such generators will be called a \emph{set of canonical generators} for $\Gamma$.

\medskip

A presentation of the above form can be encoded in the \emph{signature\/} of $\Gamma,$ as
introduced by Macbeath in \cite{macbeath}. This is a collection of symbols
and non-negative integers of the form
\begin{equation}  \label{signature}
\sigma(\Gamma) \ = \ \left(g;\,\pm;\, [m_1,\ldots,m_r];\,
\{(n_{11},\ldots,n_{1s_1}),\ldots, (n_{k1},\ldots,n_{ks_k})\}\right).
\end{equation}
The sign is ``$+$'' if there are $2g$ hyperbolic canonical generators $a_i$
and $b_i$, and ``$-$'' if there are $g$ glide reflection generators $d_i$. 
The integers $m_1,\ldots,m_r$ are called
\emph{proper periods\/}; each bracket $(n_{i1},\ldots,n_{is_i})$ is a \emph{%
period cycle\/}, and the integers $n_{ij}$ (all of which are assumed to be greater than $1$) 
are called \emph{link periods.}
An empty set of proper periods (with $r=0$) is denoted by $[-]$, 
while an empty period-cycle (with $s_i=0$) is denoted by $(-),$ 
and if there is no period cycle (with $k=0$), then the part of the signature following 
the proper periods is written as $\{-\}.$
For example, if an NEC group contains no element of finite order other than the identity, 
and hence contains no reflections and no elliptic elements, then it has
signature $(\gamma;\pm;[-];\{-\}).$

\medskip

The area of a fundamental region for an NEC group $\Gamma$ with signature (%
\ref{signature}) is $2\pi\mu(\Gamma)$, where
\begin{equation}  \label{area}
\mu(\Gamma) = \alpha g+k-2 +\sum_{i=1}^r\left(1-\frac{1}{m_i}\right)+ \frac{1%
}{2} \sum_{i=1}^k \sum_{j=1}^{s_i}\left(1-\frac{1}{n_{ij}} \right),
\end{equation}
with $\alpha=2$ if the sign is $+$, and $\alpha=1$ if the sign is $-$.
%
There exists an NEC group with the presentation encoded in signature \eqref{signature} if and only if
the expression given in \eqref{area} is positive.	

If $\Gamma^{\prime }$ is a subgroup of finite index $|\Gamma:\Gamma'|$ in $\Gamma$, then $%
\Gamma^{\prime }$ is also an \textrm{NEC} group, and its area is given by
the 
\emph{Riemann-Hurwitz formula\/}:
\begin{equation*}
\mu(\Gamma^{\prime }) = |\Gamma:\Gamma^{\prime }|\cdot\mu(\Gamma).
\end{equation*}

\medskip

The orientation-preserving elements of $\Gamma$ constitute the \emph{%
canonical Fuchsian subgroup} $\Gamma^+$. These are the elements expressible
as words in the canonical generators of $\Gamma$ containing an even number
of occurrences of the reflections $\,c_{ij}\,$ and glide reflections $\,d_i$ 
(and any number of the elliptic elements $\,x_i,\,$ and orientation-preserving elements $\,e_i.\,$) 
In particular, $\Gamma^+$ has index 1 or 2 in $\Gamma,$ and if the index is $1$
(or equivalently, $\Gamma = \Gamma^+$) then $\Gamma$ is a Fuchsian group,
while if the index is 2 then $\Gamma$ is a proper NEC group (as defined in the Introduction).

\medskip

Returning to our general question (namely about whether a proper NEC group
$\Gamma$ possess a torsion-free normal subgroup of finite index), we will make certain 
assumptions about the subgroup, in order to obtain information applicable to
different kinds of surfaces.

\medskip

A surface NEC group $\Lambda $
has signature $(\gamma ;\pm ;[-];\{(-),\overset{k}{\ldots },(-)\}).$ Now if $k>0$
then $\Lambda $ contains reflections, so is not torsion-free, and the
orbit space $\mathcal{H}/\Lambda $ is a compact Klein surface with $k$
boundary components. In this case $\Lambda $ is called a bordered surface
NEC group.
Not every NEC group contains a bordered surface NEC group, as shown in \cite%
{buj-martinez}. In fact, the following theorem gives a complete answer to the
analogue of Fenchel's conjecture for bordered surface normal subgroups of
bordered NEC groups.

\begin{theorem}[\protect\cite{buj-martinez}]
\label{bordered.surface.subgroup} Let $\Gamma$ be an NEC group with
signature \eqref{signature} with $k>0.$ Then $\Gamma$ has a bordered
surface normal subgroup with finite index in $\Gamma$ if and only if the
signature of $\Gamma$ has either an empty period cycle, or a period cycle with two
consecutive link periods equal to $2.$
\end{theorem}

In what follows, we will consider \emph{unbordered} surface NEC groups, that
is, NEC groups with signature $(\gamma ;\pm ;[-];\{-\}),$ and consider
separately the cases of sign \textquotedblleft $+$\textquotedblright\ and
\textquotedblleft $-$\textquotedblright .

\medskip

The observations below will be helpful in the next sections.

\begin{lemma}
\label{intersection.finite.index} If $N_1,\ldots, N_s$ are normal subgroups
of a group $G$ which have finite index in $G,$ then their intersection $%
N_1\cap\cdots\cap N_s$ also has finite index in $G.$
\end{lemma}

\begin{proof}
The intersection $N_1\cap\cdots\cap N_s$ is the kernel $N$ of the natural
homomorphism $f$ from $G$ to the direct product of the factor groups $G/N_i$%
, and then since the image of $f$ is finite, this normal subgroup $N$ has finite index in $G.$
\end{proof}

An easy consequence is the following.

\begin{lemma}
\label{technical.lemma} Let $H$ and $K$ be subgroups of finite index in a
group $G$ with $K\triangleleft H$ and $|G:H|=2.$ 
If $\{1,c\}$ is a transversal for $H$ in $G$, so that $G= H \cup Hc$ (and $c^2 \in H$), 
then $K\cap c^{-1}Kc$ is a normal subgroup of $G$ which also has finite index in $G.$
\end{lemma}

Theorem \ref{Fenchel's.conjecture.fuchsian} and Lemma \ref{technical.lemma}
imply the existence of finite groups with prescribed partial presentations.
We display these presentations in the following two items as notation for future reference. 

\begin{notation}
\label{G[m_1,...,m_r],G(n_1,...,n_s)}
${}$\vskip 5pt

\noindent \emph{(a) 
Let $F$ be the group generated by $%
r\ge3$ elements $x_1,\ldots,x_r$ subject to the defining relations $%
x_1^{m_1}=\cdots=x_r^{m_r}=x_1\cdots x_r=1$. Then by Theorem \ref%
{Fenchel's.conjecture.fuchsian} with $g=0,$ 
we know that $F$ has a torsion-free normal subgroup $%
K $ of finite index in $F.$ 
Hence the quotient group $F/K$ is a finite group, which we will denote by $G_{[m_1,\ldots,m_r]}$, 
and which has the partial presentation below: 
\begin{equation}  \label{G[m_1,...,m_r]}
G_{[m_1,\ldots,m_r]} = \langle\, X_1,\ldots, X_r \ | \ X_1^{m_1}=\cdots
=X_r^{m_r}=X_1\cdots X_r=\cdots =1 \,\rangle.
\end{equation}
}

\noindent \emph{(b) 
Let $\Gamma$ be the group generated by $s\ge3$ involutions $%
c_0,\ldots,c_{s-1}$ subject to the defining relations $%
(c_{i-1}c_{i})^{n_i}=1 $ for $1\le i\le s-1$ and $(c_{s-1}c_0)^{n_s}=1.$
Then the subgroup $\Gamma^+$ generated by the products $x_i=c_{i-1}c_i$ for $%
1\le i\le s-1$ and $x_s:=c_{s-1}c_0$ is a normal subgroup of index $2$ in $%
\Gamma,$ and its generators $x_1,\ldots,x_s$ satisfy the defining relations $%
x_1^{n_1}=\cdots= x_s^{n_s}=x_1\cdots x_s=1.$ 
By~Theorem~\ref%
{Fenchel's.conjecture.fuchsian}, the subgroup $\Gamma^+$ has  a
torsion-free normal subgroup $K$ of finite index in $\Gamma^+,$ and then since $%
c_0\in\Gamma \,\setminus\, \Gamma^+,$ Lemma \ref{technical.lemma} implies that the
intersection $K\cap \,c_0^{-1}Kc_0$ is a torsion-free normal subgroup of finite index in $\Gamma$. 
Hence the quotient $\Gamma/(K\cap c_0^{-1}Kc_0)$ is a finite group, which we will denote by $%
G_{(n_1,\ldots,n_s)}$, and which has the partial presentation below: 
\begin{equation}  \label{G(n_1,...,n_s)}
G_{(n_1,...,n_s)} \!= \!\langle\, C_0,..., C_{s-1} \ | \ C_i^{\,2}\!=\!
(C_0C_1)^{n_1} \!= \! \cdot\cdot\cdot \!=\!(C_{s-2}C_{s-1})^{n_{s-1}} \!=\!
(C_{s-1}C_0)^{n_s} = \cdot\cdot\cdot =1\,\rangle.
\end{equation}
In particular, for notational convenience, in the special case where $s = 1$ we will sometimes 
denote the cyclic group of order $2$ by $%
G_{(-)}$, as this equals $\langle\, C_0 \ | \  C_0^{\,2}=1 \,\rangle$. 
}
\end{notation}

\section{Orientation-preserving subgroups of proper NEC\ groups}
\label{orientation.preserving.subgroup}

This case is easy:

\begin{theorem}
\label{orientable.surface.subgroup} 
Let $\Gamma$ be a proper NEC group with
signature $\eqref{signature}.$ Then $\Gamma$ has a torsion-free normal
subgroup of finite index consisting of orientation-preserving elements.
\end{theorem}

\begin{proof}
Applying Theorem \ref{Fenchel's.conjecture.fuchsian} to the canonical
Fuchsian subgroup $\Gamma^+$ of $\Gamma$ gives the existence of a normal
subgroup $K$ of finite index in $\Gamma^+$ containing no non-trivial element
of finite order, and then if we let $\{1,c\}$ be a transversal for $\Gamma^+$ in $\Gamma$, 
then by Lemma \ref{technical.lemma} the intersection $\Lambda:=K\cap c^{-1}Kc$ 
is also a normal subgroup of finite index in $\Gamma$ containing no non-trivial element 
of finite order, and being a subgroup of $\Gamma^+$, consists of orientation-preserving elements. 
\end{proof}

\section{Orientation-reversing subgroups of proper NEC\ groups}
\label{orientation.reversing.subgroup} 

In the rest of the paper, we address the question of existence of a torsion-free normal subgroup 
of finite index in a proper NEC group $\Gamma$, containing orientation-reversing elements.
Observe that these orientation-reversing elements must be glide reflections, 
as all hyperbolic reflections (the only other kind of orientation-reversing elements) have finite order.

We distinguish cases according to the sign in the signature of $\Gamma.$

\subsection{sign({\protect\boldmath$\Gamma)=``-$''} \ (with {\protect \boldmath$g>0)$}}\label{subsection.sign(Gamma)=-}

We first consider the case $k=0,$ where $\Gamma$ has no hyperbolic reflections, 
so there is no torsion element in $\Gamma\,\setminus\,\Gamma^+.$

\subsubsection{\protect\boldmath$\protect\sigma(\Gamma)=(g;-;[m_1,\ldots,m_r];\{-\})$}

\label{subsection.unbordered.non-orientable}

\begin{theorem}
\label{theorem.unbordered.non-orientable} Let $\Gamma$ be an NEC group with
signature $\sigma(\Gamma)=(g;-;[m_1,\ldots,m_r];\{-\}).$ Then $\Gamma$
has a torsion-free normal subgroup 
of finite index that contains orientation-reversing elements.
\end{theorem}


\begin{proof}
First, suppose that $r\ge3.$ In this case, we take a finite group $%
G=G_{[m_1,\ldots,m_r]}$ with partial presentation $\eqref{G[m_1,...,m_r]},$
and define $\phi:\Gamma\to G$ by setting $\phi(d_i)=1$ for $1\le i\le g$ and $%
\phi(x_i)=X_i$ for $1\le i\le r.$ Then this gives a homomorphism whose kernel is a torsion-free 
normal subgroup of finite index $|G|$ in $\Gamma$, and contains the orientation-reversing 
elements $d_i$ (for $1 \le i \le g$).

Next suppose that $r=2.$ Here we take a finite group $G=G_{[m_1,m_2,3]}$
generated by two elements $u$ and $v$ of orders $m_1$ and $m_2$ respectively
such that their product $uv$ has order $3$, and define $\phi:\Gamma\to G$ by setting $%
\phi(x_1)=u,$ $\phi(x_2)=v,$ $\phi(d_1)=uv$ and $\phi(d_i)=1$ for any 
remaining $d_i$ (with $i > 1$).  Then this gives a homomorphism whose kernel 
is a torsion-free normal subgroup of finite index $|G|$ in $\Gamma$, and contains the 
orientation-reversing element $d_1^{\,3}.$

Finally suppose that $r=1$, so that $g\ge2$ (in order to ensure that the area of a fundamental region 
for $\Gamma$ is positive). In this case take a finite group $G=G_{[3,3,m_1]}$ generated by two 
elements $u$ and $v$ of order $3$ such that their product $uv$ has order $m_1,$ 
and define $\phi:\Gamma\to G$ by setting $%
\phi(d_1)=u,$ $\phi(d_2)=v,$ $\phi(x_1)=vu$ and $\phi(d_i)=1$ for any $i > 2$.
Again this gives a homomorphism whose kernel is a torsion-free normal subgroup of finite index $|G|$ 
in $\Gamma$, and contains the orientation-reversing element $d_1^{\,3}.$
\end{proof}

Note that the above theorem can also be quickly deduced from 
\cite[Proposition 3.1]{izquierdo.procLondonMathSoc}.

\begin{remark}
\emph{If $r=0$ then $\sigma(\Gamma)=(g;-;[-];\{-\})$, with $g\ge3$, and then $%
\Gamma$ itself is torsion-free, and obviously contains orientation-reversing elements, 
so Theorem \ref{theorem.unbordered.non-orientable} is trivially true in that case. 
It is now quite natural to ask whether such an NEC group $\Gamma$ contains
any proper normal subgroups satisfying Theorem \ref{theorem.unbordered.non-orientable}.
It is easy to see that the answer is affirmative.  For let $G$ be any finite group with a non-trivial element $u$,  
and define $\phi:\Gamma\to G$ by setting $\phi(d_1)=u,$ $\phi(d_2)=u^{-1}$ and $\phi(d_i)=1$ for all $i\ge3$. Then this gives a homomorphism whose kernel is a proper normal torsion-free subgroup of $\Gamma$ 
containing the orientation-reversing element $d_3$.} 
\end{remark}

\medskip

We now proceed to the case where $k=1.$

\subsubsection{\protect\boldmath$\protect\sigma(\Gamma)=(g;-;[m_1,%
\ldots,m_r];\{(n_1,\ldots,n_s)\})$}

\label{subsection.bordered.non-orientable.k=1}

\begin{theorem}
\label{theorem.bordered.non-orientable.k=1} Let $\Gamma$ be an NEC group
with signature $\sigma(\Gamma)=(g;-;[m_1,\ldots,m_r];\{(n_1,\ldots,n_s)\}).$
Then $\Gamma$ has a torsion-free normal subgroup of finite index in $%
\Gamma$ that contains orientation-reversing elements.
\end{theorem}


\begin{proof}
First, if $r=s=0$ then $g \ge 2$, and we may define a homomorphism 
$\phi:\Gamma\to  \mathrm{C_2} = \langle u \rangle$ by 
setting $\phi(d_i)=1$ for $1\le i\le g,\,$ and $\phi(c_0)=u\,$ and $\,%
\phi(e_1)=1.$  Then $\ker\phi$ is torsion-free, has index $2$ in $\Gamma$, and contains 
the orientation reversing elements $d_i$. 

Hence we may assume from now on that $r+s > 0$.

Now let $\{d_1,\ldots,d_g,x_1,\ldots,x_r,c_0,\ldots,c_s,e\}$ be a canonical set
of generators for $\Gamma,$ with $c_s=ec_0e^{-1}.$ Then the assignment $%
\psi:\Gamma\to\mathrm{C}_2=\langle u\rangle$ given by $\psi(c_{i})=u$ for $%
0 \le i \le s$ and $\psi(x)=1$ for every remaining canonical generator $x$ of $%
\Gamma$ clearly defines a homomorphism, with orientation-preserving kernel. 
By \cite[Theorems 2.1.3 and 2.2.4]{begg} and the Riemann-Hurwitz
formula, it is easy to see that $\ker\psi$ has signature $%
(2g;-;[m_1,m_1,\ldots,m_r,m_r,$ $n_1,\ldots,n_s];\{-\}),$ and then using conjugation of the generators, 
we can rearrange the proper periods so that it has signature $(2g;-;[m_1,m_2,%
\ldots,m_r, n_1,n_2,\ldots,n_s,$ $m_r,m_{r-1},\ldots,m_1];\{-\}).$ In that case, 
a set of generators for $\ker\psi$ corresponding to the signature may be taken as follows:

\begin{itemize}
\item $d^{\,\prime }_i=d_{g+1-i}^{\,-1}$ for $1 \le i \le g$, 

\item $d^{\,\prime }_{g+i} = c_0d_ic_0$ for $1 \le i \le g$, 

\item $x^{\,\prime }_i=c_0x_ic_0\,$  (so that $%
x^{\,\prime }_i$ has order $m_i)\,$ for $1 \le i \le r$ 

\item $x^{\,\prime }_{r+j}=c_{j-1}c_j\,$  (so that $%
x^{\,\prime }_{r+j}$ has order $n_j)\,$ for $1 \le j \le s$, \ and 

\item $x^{\,\prime }_{r+s+i}= (x_{r+1-i})^{-1}\,$  (so that $x^{\,\prime }_{r+s+i}$ has order $m_{r+1-i})\,$ 
for $1 \le i \le r\,;$ 
\end{itemize}

\noindent 
and then it is easy to see that these generators satisfy the relations
\begin{equation}  \label{defining.relations}
({x^{\,\prime }_i})^{m_i}=({x^{\,\prime }}_{\!\!r+s+i})^{m_i}=({x^{\,\prime }}%
_{\!\!r+j})^{n_j}=1 \  \mbox{ for all $i$ and $j,\,$, and } \, {%
(d_1^{\,\prime }})^2\cdots ({d_{2g}^{\,\prime }})^2 \, x^{\,\prime }_1\cdots
x^{\,\prime }_{2r+s}=1.
\end{equation}

\noindent 
In fact, this set $\{d^{\prime }_1,\ldots,d^{\prime }_{2g},x^{\prime
}_1,\ldots,x^{\prime }_{2r+s}\}$ constitutes a canonical set of generators for 
$\ker\psi,$ by Remark~\ref{complete.system.of.generators.and.relations} below.

Now by Theorem \ref{theorem.unbordered.non-orientable}, the NEC group $\ker\psi$ 
has a torsion-free normal subgroup $K$ of finite index containing orientation-reversing elements,  
and by Lemma \ref%
{technical.lemma} it follows that the intersection $N=K\cap c_0 Kc_0$ is a normal subgroup
of finite index in $\Gamma.$ Then since $N$ is clearly torsion-free, all we need to do is show 
that $N$ contains orientation-reversing elements. 
To do that, we note that $r+2s \ge r+s \ge 1$, and show that there exists a glide reflection $%
d\in K$ such that its conjugate $c_0dc_0$ also belongs to $K.$ 

If $2r+s>2$, then by the way in which $K$ is obtained in the proof 
of Theorem \ref{theorem.unbordered.non-orientable}, we may suppose that all of the canonical 
generating glide reflections $d^{\,\prime }_i$ of $\ker\psi$ are contained in $K$, 
and then by the above expression of generators of $\ker\psi$ in terms of the canonical generators 
of $\Gamma$, we find that 
$c_0 d^{\,\prime }_i c_0 = c_0 d_{g+1-i}c_0 =(d^{\,\prime}_{2g+1-i})^{-1},$ 
which is the inverse of a glide reflection in $K,$ and hence lies in $K$, for all $i$. 

If $2r+s=2,$ then by the proof of Theorem \ref{theorem.unbordered.non-orientable}, 
we may suppose that $K$ contains $%
(d^{\,\prime }_1)^3$ and all the remaining generating glide reflections $%
d^{\,\prime }_2,\ldots,d^{\,\prime }_{2g}.$ 
This time we find that 
$c_0(d^{\,\prime}_1)^3c_0 = c_0(d_g)^{-3}c_0 = (c_0d_gc_0)^{-3} =(d^{\,\prime }_{2g})^{-3}, $
which is a glide reflection lying in $K.$  

Finally, if $2r+s=1$ then we may suppose that $K$ contains $(d^{\,\prime}_1)^3,$ $(d^{\,\prime }_2)^3$ 
and all of the remaining glide reflections $d^{\prime }_3,\ldots,d^{\prime }_{2g}$ (if $g > 1$), 
and then just as above, we find that the conjugate $c_0(d^{\,\prime }_1)^3c_0$ lies in $K,$ 
completing the proof.
\end{proof}

${}$\\[-42pt] ${}$

\begin{remark}
\label{special cases} 
\emph{Theorem \ref{theorem.bordered.non-orientable.k=1} is clearly valid in the special cases 
of NEC groups $\Gamma$ with signatures $(g;-;[-];\{(n_1)\}),$ for which $2r+s=1$, 
and $(g;-;[-];\{(n_1,n_2)\}),$ for which $2r+s=2.$ }
\end{remark}

\begin{remark}
\label{complete.system.of.generators.and.relations} \emph{The edges of a
fundamental region $F$ for $\Gamma$ can be labelled according to the action 
of the elements of $\Gamma$; see \cite{wilkie} or \cite[Chap. 0]%
{begg}. In this case we can choose $F$ to be a hyperbolic polygon with $%
2g+2r+s+3$ edges, so that its labelling is $\delta_1\delta_1^{\,*}\ldots%
\delta_{g}\delta_{g}^{\,*}\,\xi_1\xi_1^{\,\prime }\ldots\xi_r\xi_r^{\,\prime
}\,\varepsilon\,\gamma_0\ldots\gamma_s\,\varepsilon^{\prime},$ 
where $d_i$ swaps each $\delta_i$ with $\delta_i^{\,*}$ (for $1 \le i \le g$), 
and $x_i$ swaps each $\xi_i$ with $\xi^{\,\prime}_i$ (for $1 \le i \le r$), 
and $e$ swaps $\varepsilon$ with $\varepsilon^{\prime},$ and $c_i$ fixes  each $\gamma_i.$
A fundamental region $\Psi$ for $\ker\psi$ can then be obtained by glueing
together two copies of $F.$ 
This gives a labelling of the edges of $\Psi,$ and then the fundamental region (and it labelling) 
can be modified by cutting and glueing different parts of $\Psi$. 
Applying this procedure in a suitable way, we can obtain a fundamental region for $\ker\psi$
with $4g+4r+2s$ edges labelled 
as $%
\hat\delta_1\hat\delta_1^{\,*}\ldots\hat\delta_{2g}\hat\delta_{2g}^{\,*}\,\mu_1\mu^{%
\,\prime }_1\ldots\mu_r\mu^{\,\prime }_r \,\eta_1\eta^{\,\prime
}_1\ldots\eta_s\eta^{\,\prime }_s\nu_r\,\nu^{\,\prime }_r\ldots\nu_1\nu^{\,\prime
}_1$, 
where $d^{\,\prime }_i=d_{g+1-i}^{\,-1}$ swaps $\hat\delta_i$ with $%
\hat\delta^*_i$ for $1 \le i \le g,$ 
and $d^{\,\prime }_{g+i} = c_0d_ic_0$ swaps 
$\hat\delta_{g+i}$ with $\hat\delta^{\,*}_{g+i}$ for $1 \le i \le g$, 
and $x^{\,\prime }_i=c_0x_ic_0$ swaps $\mu_i$ with $\mu^{\,\prime }_i$ for $1 \le i \le r$,
and $x^{\,\prime }_{r+i}=c_{i-1}c_i$ swaps $\eta_i$ with $\eta^{\,\prime }_i$ for $1 \le i \le s,$ 
and $x^{\,\prime }_{r+s+i}= (x_{r+1-i})^{-1}$ swaps $\nu_{r+1-i}$ with $\nu^{\,\prime }_{r+1-i}$ 
for $1 \le i \le r$. 
In this situation, a theorem by Macbeath \cite%
{macbeath.annals} shows that the relations \eqref{defining.relations}
constitute a complete set of defining relations  for the group $\ker\psi$ in terms of its generating set 
$\{d^{\,\prime}_1,\ldots,d^{\,\prime }_{2g},x^{\,\prime }_1,\ldots,x^{\,\prime }_{2r+s}\}$. 
 }
\end{remark}

\medskip

Finally in this subsection~\ref{subsection.sign(Gamma)=-}, we consider the case $k\ge2.$

\subsubsection{\protect\boldmath$\protect\sigma(\Gamma)=(g;-;[m_1,%
\ldots,m_r]; \{\mathcal{C}_1,\ldots,\mathcal{C}_k\}) \ \mbox{with} \ k\ge2$}

\begin{theorem}
\label{theorem.non-orientable.k>1} Let $\Gamma$ be an NEC group with
signature $(g;-;[m_1,\ldots,m_r];\, \{\mathcal{C}_1,\ldots,\mathcal{C}_k\})$, 
such that $k\ge2.$ 
Then $\Gamma$ has a torsion-free normal subgroup of finite index in $%
\Gamma$ that contains orientation-reversing elements.
\end{theorem}

\begin{proof}
First, suppose $k=2$ and $r=s_1=s_2=0$, so that $\Gamma$ has signature $(g;-;[-];\{(-),(-)\}).$ 
Then there exists a homomorphism $\Gamma\to \mathrm{C}_2=\langle u\rangle$ given by 
the assignments $d_i\mapsto 1, \ c_{10}\mapsto u, \ c_{20}\mapsto u, \ e_1\mapsto 1\,$ and $\,e_2\mapsto 1$,  
with kernel a torsion-free normal subgroup of index $2$ in $\Gamma$ which contains 
the orientation-reversing element(s) $d_i$. 

Next, if we are not in the situation where $k=2$ and $r=s_1=s_2=0$, consider an NEC group $\Gamma_0$ 
with signature $\sigma_0=(0;+;[m_1,\ldots,m_r];$ $\{\mathcal{C}_1,\ldots,\mathcal{C}_k\})$, 
obtainable by ``keeping the torsion part'' of the given signature $\sigma(\Gamma)$ 
and setting $g$ to be $0$.  We can do this, as the assumption about $k$, $r$, $s_1$ and $s_2$ 
implies that a fundamental region for this group $\Gamma_0$ has positive area.

Applying Theorem \ref{Fenchel's.conjecture.fuchsian} to the canonical
Fuchsian subgroup $\Gamma_0^+$ of $\Gamma_0$ gives the existence of a torsion-free normal
subgroup $K$ of finite index in $\Gamma_0^+$, and then letting $\{1,c\}$ be a transversal for 
$\Gamma_0^+$ in $\Gamma_0$, we find by Lemma \ref{technical.lemma} that the
intersection $\Lambda:=K\cap c^{-1}Kc$ is a torsion-free normal subgroup of finite
index in $\Gamma_0$. 
It follows that the factor group $\Gamma_0/\Lambda$ is a finite group $G$
generated by elements $X_i:=\Lambda x_i$ for $1\le i\le r,$ and $%
C_{ij}:=\Lambda c_{ij}$ for $0\le j \le s_i$ and $1\le i\le k$, and $%
E_i:=\Lambda e_i$ for $1\le i \le k,\,$ subject to the relations
\begin{eqnarray*}
& & \!\!X_i^{m_i}=1 \ \hbox{ for } 1\le i\le r, \quad
C_{ij-1}^{\,2}=C_{ij}^{\,2}=(C_{ij-1}C_{ij})^{n_{ij}}=1 \ \hbox{ for } 0\le j \le s_i \,\hbox{ and }\, 1\le i\le k,  \\
& & \quad E_iC_{i0}E_i^{-1} = C_{is_i} \ \hbox{ for } 1\le i\le k, \ \ \hbox{ and} \ \ X_1\cdots X_r\,E_1\cdots E_k= 1,
\end{eqnarray*}
plus other relations which make $G = \Gamma_0/\Lambda$ finite. 
We can now define $%
\phi:\Gamma\to G$ by setting $\phi(d_i)=1$ for $1\le i\le g,$ and $\phi(x_i)=X_i$ for $%
1\le i\le r,$ and $\phi(c_{ij})=C_{ij}$ for $0\le j \le s_i$ and $1\le i\le k$, and $%
\phi(e_i)=E_i$ for $1\le i\le k,$ giving a homomorphism whose kernel is a torsion-free 
normal subgroup of finite index $|G|$ in $\Gamma$, and contains the orientation-reversing element(s) $d_i$.
\end{proof}

\subsection{sign({\protect\boldmath$\Gamma)=``+$''}   \ (with {\protect \boldmath$k\ge1)$}}\label{subsection.sign(Gamma)=+}

We first consider the case $g>0.$


\begin{theorem}
\label{theorem.non-orientable.g>0} Let $\Gamma$ be a proper NEC group with
signature $(g;+;[m_1,\ldots,m_r];\, \{\mathcal{C}_1,\ldots,\mathcal{C}_k\})$
for some $g>0.$ Then $\Gamma$ has a torsion-free normal subgroup of finite index 
that contains orientation-reversing elements.
\end{theorem}

\begin{proof}
First we assume that $\mu(\Gamma)>2g.$ Then there exists an NEC group $%
\Gamma_0$ with signature $\sigma_0=(0;+;[m_1,\ldots,m_r];$ $\{\mathcal{C}%
_1,\ldots,\mathcal{C}_k\})$, noting that our assumption ensures that  a fundamental region 
for this group $\Gamma_0$ has positive area.
Then just as in the proof of Theorem~\ref{theorem.non-orientable.k>1}, there exists a torsion-free 
normal subgroup $\Lambda$ of finite index in $\Gamma_0$ such that the quotient $\Gamma_0/\Lambda$ 
is a finite group $G$ generated by elements 
$\,X_i:=\Lambda x_i$ for $1\le i\le r,\,$ and $C_{ij}:=\Lambda c_{ij}$ for $0\le j \le s_i$ and $1\le i\le k,\,$ 
and $\,E_i:=\Lambda e_i$ for $1\le i \le k,$
subject to the relations
\begin{eqnarray*}
& & \!\!X_i^{m_i}=1 \ \hbox{ for } 1\le i\le r, \quad
C_{ij-1}^{\,2}=C_{ij}^{\,2}=(C_{ij-1}C_{ij})^{n_{ij}}=1 \ \hbox{ for } 0\le j \le s_i \,\hbox{ and }\, 1\le i\le k,  \\
& & \quad E_iC_{i0}E_i^{-1} = C_{is_i} \ \hbox{ for } 1\le i\le k, \ \ \hbox{ and} \ \ X_1\cdots X_r\,E_1\cdots E_k= 1,
\end{eqnarray*}
plus other relations which make $G = \Gamma_0/\Lambda$ finite. 
We can now define $%
\phi:\Gamma\to G$ by setting $\phi(a_1)=C_{10},$ $\phi(x_i)=X_i$ for $1\le i\le r,$ and $%
\phi(c_{ij})=C_{ij}$ for $0\le j \le s_i$ and $1\le i\le k$, and $\phi(e_i)=E_i$ for $1\le i\le k$, 
and $\phi(x)=1$ for all of the remaining canonical generators $a_2,\ldots,a_g,b_1,\ldots,b_g$ of $\Gamma.$ 
This gives a homomorphism whose kernel is a torsion-free normal subgroup of finite index $|G|$ in $\Gamma$, 
containing the orientation-reversing element $a_1c_{10}.$ 

\medskip

For the rest of the proof, we assume instead that $\mu(\Gamma) \le 2g.$

\smallskip

In that case, \ $k-2+\sum_{i=1}^r\left(1-\frac{1}{m_i}\right)+\frac12\sum_{i=1}^k
\sum_{j=1}^{s_i}\left(1-\frac{1}{n_{ij}} \right) = \mu(\Gamma) - 2g \le 0$, 
and so clearly $k \le 2$.  
Also if $k=2$ then $\sigma(\Gamma) = (g;+;[-];\{(-),(-)\})$, while otherwise $k = 1$ and then 
setting $s= s_1$ and $n_j = n_{1j}$ for $1 \le j \le s$, we have  
$\sum_{i=1}^r\left(1-\frac{1}{m_i}\right)+\frac12 \sum_{j=1}^{s}\left(1-\frac{1}{n_{j}} \right) \le1,$
and then since $1-1/m_i\ge1/2$ for each $i$ and $1-1/n_j\ge1/2$ for each $j$, we
find that $\frac{r}{2}+\frac{s}{4} \le 1,$ so $2r+s\le4.$ 
Accordingly, we have the following possibilities: 
\\[+8pt]
\begin{tabular}{cl}
$\bullet$ & $k = 2,\, r=0,\, s_1= s_2 = 0\,$ and $\,\sigma(\Gamma) = (g;+;[-];\{(-),(-)\});$
\\[+6pt]
$\bullet$ & $k = 1,\, r=2,\, s=0\,$ and $\,\sigma(\Gamma) = (g;+;[2,2];\{(-)\});$
\\[+6pt]
$\bullet$ & $k = 1,\, r=1,\, s=2\,$ and $\,\sigma(\Gamma) = (g;+;[2];\{(2,2)\});$
\\[+6pt]
$\bullet$ & $k = 1,\, r=2,\, s=0\,$ and $\,\sigma(\Gamma) = (g;+;[2,2];\{(-)\});$
\\[+6pt]
$\bullet$ & $k = 1,\, r=1,\, s=2\,$ and $\,\sigma(\Gamma) = (g;+;[2];\{(2,2)\});$
\\[+6pt]
$\bullet$ & $k = 1,\, r=1,\, s=1$, and \\ & \ $\sigma(\Gamma) = (g;+;[2];\{(n)\}),$ $(g;+;[3];\{(3)\}),$ $%
(g;+;[3];\{(2)\})$ or $(g;+;[4];\{(2)\});$
\\[+7pt]
$\bullet$ & $k = 1,\, r=1,\, s=0\,$ and $\,\sigma(\Gamma) = (g;+;[m];\{(-)\});$
\\[+6pt]
$\bullet$ & $k = 1,\, r=0,\, s=4\,$ and $\,\sigma(\Gamma) = (g;+;[-];\{(2,2,2,2)\});$
\\[+6pt]
$\bullet$ & $k = 1,\, r=0,\, s=3\,$ and $\,\sigma(\Gamma) = (g;+;[-];\{(n_1,n_2,n_3)\})$ \\ & \ where $%
(n_1,n_2,n_3)=$ $(2,2,m),$ $(2,3,3),$ $(2,3,4),$ $(2,3,5),$ $(2,3,6),$ $%
(2,4,4)$ or $(3,3,3);$
\\[+6pt]
$\bullet$ & $k = 1,\, r=0,\, s=2\,$ and $\,\sigma(\Gamma) = (g;+;[-];\{(n_1,n_2)\});$
\\[+6pt]
$\bullet$ & $k = 1,\, r=0,\, s=1\,$ and $\,\sigma(\Gamma) = (g;+;[-];\{(n)\});$
\\[+6pt]
$\bullet$ & $k = 1,\, r=0,\, s=0\,$ and $\,\sigma(\Gamma) = (g;+;[-];\{(-)\}).$
\end{tabular}

\medskip
For most of these possibilities, we can define a homomorphism $\phi $ from $\Gamma$
to a finite group $G$ containing an involution $u$ such that 
$\phi (a_{1})=\phi(c_{10}) =u$ and $\phi(a_{i})=1$  for $2\leq i\leq g$ and $\phi (b_{i})=1$ for $1 \le i \le g$, 
and the effect of $\phi$ on other canonical generators of $\Gamma$ is arranged so that $K = \ker \phi$ is 
torsion-free.  Then $K$ has finite index in $\Gamma$ (dividing $|G|$), and contains the orientation-reversing element $a_{1}c_{10}$.
We do this in those cases as follows:
\\[+7pt]
\begin{tabular}{cl}
\!\!(a)\!\!\! & If $\sigma (\Gamma )=(g;+;[-];\{(-),(-)\})$, take $G=\mathrm{C}_{2}=\langle u \rangle$, \\ & 
and then set also  $\phi (e_{1})=\phi (e_{2})=1$ and $\phi (c_{20})=u\,;$
\\[+5pt]
\!\!(b)\!\!\! & If $\sigma (\Gamma )=(g;+;[2,2];\{(-)\})$, 
take $G=\mathrm{C}_{2}\times \mathrm{C}_{2}=\langle\, u,v \ | \ u^{2}=v^{2}=[u,v]=1\rangle$, 
 \\ & and then set also $\phi(x_{1})=v$ and $\phi(x_{2})=u,$ and $\phi (e_{1})=uv\,;$  
\\[+5pt]
\!\!(c)\!\!\! & If $\sigma (\Gamma )=(g;+;[2];\{(2,2)\})$, 
take $G=\mathrm{C}_{2}\times \mathrm{C}_{2}=\langle\, u,v \ | \ u^{2}=v^{2}=[u,v]=1\,\rangle $, 
\\ & and then set also $\phi (x_{1})=u,$ and $\phi (e_{1})=u,$  and $\phi (c_{11})=v\,;$
\\[+5pt]
\!\!(d)\!\!\! & If $\sigma (\Gamma )=(g;+;[2];\{(n)\})$, 
take $G=\mathrm{D}_{2n}=\langle\, u,v \ | \  u^{2}=v^{2}=(uv)^{2n}=1\,\rangle $, 
 \\ & and then set also  $\phi (x_{1})=v$, and $\phi (e_{1})=v$, and $\phi(c_{11}) = vuv\,;$
\\[+5pt]
\!\!(e)\!\!\! & If $\sigma (\Gamma )=(g;+;[3];\{(3)\})$, 
take $G=\mathrm{S}_4 = \langle\, u \ | \ u^{2}=v^{3}=(uv)^4=[u,v]^3=1\,\rangle$ 
via $u = (1,2)$ \\ & and $v = (2,3,4)$, 
and then set also $\phi (x_{1})=v^{-1}$, and $\phi (e_{1})=v$, and $\phi(c_{11}) = v^{-1}uv\,;$
\\[+5pt]
\!\!(f)\!\!\! & If $\sigma (\Gamma )=(g;+;[3];\{(2)\})$, 
take $G=\mathrm{A}_{4} = \langle\, u,v \ | \ u^{2}=v^{3}=[u,v]^2=1\,\rangle$ 
via $v = (1,2,3)$ \\ & and $u = (1,2)(3,4)$, 
and then set also $\phi (x_{1})=v^{-1}$, and $\phi (e_{1})=v$, and $\phi(c_{11}) = v^{-1}uv\,;$
\\[+5pt]
\!\!(g)\!\!\! & If $\sigma (\Gamma )=(g;+;[4];\{(2)\})$, 
take $G=\mathrm{D}_{4}=\langle\, u,v \ | \ u^{2}=v^{4}=(uv)^{2}=1\,\rangle$, \\ & 
and then set also $\phi (x_{1})=v^{-1}$, and $\phi (e_1)=v$, and $\phi(c_{11}) = v^{-1}uv = uv^2\,;$
\\[+5pt]
\!\!(h)\!\!\! & If $\sigma (\Gamma )=(g;+;[m];\{(-)\})$, 
take $G=\mathrm{C}_{2}\times \mathrm{\ C}_{m}=\langle\, u,v \ | \ u^{2}=v^{m}=[u,v]=1\,\rangle $, \\ & 
and then set also $\phi (x_{1})=v$ and $\phi (e_1)=v^{-1}\,;$
\\[+5pt]
\!\!(i)\!\!\! & If $\sigma (\Gamma )=(g;+;[-];\{(2,2,2,2)\})$, 
take $G=\mathrm{C}_{2}\times \mathrm{C}_{2}=\langle\, u,v \ | \ u^{2}=v^{2}=[u,v]=1\,\rangle $, \\ & 
and then set also $\phi (e_1)=1,$ and $\phi (c_{11})=v,$ $\,\phi (c_{12})=u,$ $\,\phi (c_{13})=v$ 
and $\,\phi (c_{14})=u\,;$
\\[+5pt]
\!\!(j)\!\!\! & If $\sigma (\Gamma )=(g;+;[-];\{(n_{1},n_{2},n_{3})$, 
take $G$ to be the (finite) spherical full triangle group \\ & $\Delta(n_{1},n_{2},n_{3}) 
\cong {\rm D}_m$, ${\rm A}_4$, ${\rm S}_4$ or ${\rm A}_5$  
when $(n_{1},n_{2},n_{3}) = (2,2,m),$ $(2,3,3),$ $(2,3,4),$ or $(2,3,5)$, \\ & 
or a finite smooth quotient of the (infinite) full Euclidean triangle group $\Delta(n_{1},n_{2},n_{3})$ 
\\ & when $(n_{1},n_{2},n_{3}) = (2,3,6)$, $(2,4,4)$ or $(3,3,3)$, 
generated by three involutions $u,v$ and $w$ \\ & such that $uv, vw$ and $wu$
have orders $n_{1},$ $n_{2}$ and $n_{3}$ respectively, see \eqref{G(n_1,...,n_s)}, 
and then set also \\ & $\phi(e_1)=1$, and $\phi (c_{11})=v,\,$ $\phi (c_{12})=w$ and $\phi (c_{13})=u\,;$
\\[+5pt]
\!\!(k)\!\!\! & If $\sigma (\Gamma )=(g;+;[-];\{(-)\})$, take $G=\mathrm{C}_{2}=\langle u \rangle$,  
and then set also $\phi (e_1)=1$.
\end{tabular}
 
\medskip\smallskip
In the two remaining cases, we can do the same, but take $\phi(b_1)$ as a non-trivial  
element of $G$, chosen to ensure that the relation 
$[a_1,b_1] \ldots [a_g,b_g]\, x_1 \ldots x_r \, e_1 \ldots e_k= 1$ is preserved: 
\\[+7pt]
\begin{tabular}{cl}
\!\!(l)\!\!\! & If $\sigma (\Gamma )=(g;+;[-];\{(n_1,n_2)\})$, 
take $G=\mathrm{D}_{4n_{1}n_{2}}=\langle\, u,v \ | \ u^{2}=v^{2}=(uv)^{4n_{1}n_{2}}=1\,\rangle$, \\ & 
and then set also $\,\phi(b_1) = (uv)^{n_2-n_1}$, and $\phi(e_1) = (uv)^{2(n_1-n_2)}$, 
and $\phi(c_{11}) = u(uv)^{4n_2}$ \\ & and $\phi(c_{12}) = (uv)^{4(n_1-n_2)}u\,;$
\\[+5pt]
\!\!(m)\!\!\! & If $\sigma (\Gamma )=(g;+;[-];\{(n_1)\})$, 
take $G=\mathrm{D}_{4n_1}=\langle\, u,v \ | \ u^{2}=v^{2}=(uv)^{4n_1}=1\,\rangle $, \\ & 
and then set also $\phi (b_{1})=uv,$ and $\phi (e_1)=(uv)^{-2}$, and $\phi(c_{11}) = u(uv)^{4}$.  
\end{tabular}

\smallskip

\vskip -13pt
\end{proof}


Next, we consider the case where $g=0,$ that is, 
for $\sigma (\Gamma)=(0;+;[m_{1},\ldots ,m_{r}];\{\mathcal{C}_{1},\ldots ,\mathcal{C}_{k}\}).$
This is by far the most difficult of all the cases we have considered, 
and we have resolved the NEC version of Fenchel's conjecture for only some particular subcases.

\smallskip
To explain this further,  we define the following parameters:
\\[+3pt]
\begin{tabular}{cl}
$\bullet$ & $k_{0}$ is the number of empty period cycles $\mathcal{C}=(-),$ \\
$\bullet$ & $k_{1}$ is the number of period cycles $\mathcal{C}=(n)$ with exactly one link period, \\
$\bullet$ & $k_{2}$ is the number of period cycles $\mathcal{C}=(n_1,n_2)$ with exactly two
link periods, \ and \\ 
$\bullet$ & $k_{3}$ is the number of period cycles $\mathcal{C}=(n_{1},\ldots ,n_{s})$
with three or more link periods, so $s\geq 3.$
\end{tabular}

We can resolve Fenchel's conjecture in the case of NEC groups having signatures 
with $g = 0$ and $k_{0}+k_{3}\geq 2,$ as follows. 



\begin{theorem}
\label{theorem.non-orientable.g=0.k_0+k_3>1} 
Let $\Gamma$ be a proper NEC group with
signature $(0;+;[m_1,\ldots,m_r];\, \{\mathcal{C}_1,\ldots,\mathcal{C}_k\})$, 
and let $k_0, k_1,k_2$ and $k_3$ be the parameters defined above. 
If $k_0+k_3\ge2$, then $\Gamma$ has  a torsion-free normal subgroup of finite index 
that contains orientation-reversing elements.
\end{theorem}

\begin{proof}
Without loss of generality, we may rearrange the period cycles so that the first two are 
$\mathcal{C}_1=(n_{11},\ldots,n_{1s_1})$ and $\mathcal{C}_2=(n_{21},\ldots,n_{2s_2})$, 
with $s_1=0$ or $s_1\ge3$ and $s_2=0$ or $s_2\ge3$. 

Now let $\{c_{i0},\ldots,c_{is_1},e_i\}$ be the canonical set of generators associated 
with the cycle $\mathcal{C}_i$ (with $c_{is_i}=e_ic_{i0}e_i^{-1})$, 
and let $G_{(n_{i1},\ldots,n_{is_i})}$ be a finite group with (partial) presentation
\begin{equation*}
\langle\, C_{i0},\ldots, C_{is_i-1} \ | \ C_{ij}^{\,2}\!=\!(C_{i0}C_{i1})^{n_{i1}} \!
= \! \cdots \!=\!(C_{is_i-2}C_{is_i-1})^{n_{is_i-1}} \!=\! (C_{is_i-1}C_{i0})^{n_{is_i}}= \cdots =1 \,\rangle, 
\end{equation*}
such that each $C_{ij}$ has order $2$ (for $0 \le j \le s_i-1$), 
each $C_{ij-1}C_{ij}$ has order $n_{ij}$ (for $1 \le j \le s_i-1$) 
and $C_{is_i-1}C_{i0}$ has order $n_{is_i}$, if $s_i\ge3$ (see~\eqref{G(n_1,...,n_s)}), 
or more simply ${\rm C}_2 = \langle\, C_0 \ | \ C_0^{\,2}=1 \,\rangle$ if $s_i=0,$  for $i = 1$ and $2$.
Note here that when $s_i > 0$, the subgroups $\langle C_{ij-1},C_{ij} \rangle$ 
and $\langle C_{is_i-1},C_{i0} \rangle$ are dihedral, of orders $2n_{ij}$ (for $1 \le j \le s_i-1$) 
and $2n_{is_i}$, for each $i$. 

We may define $\phi_1:\Gamma\to G_{(n_{11},\ldots,n_{1s_1})}$ by setting $%
\phi_1(c_{1j})=C_{1j}$ for $0\le j\le s_1-1,$ and $\phi(c_{1s_1}) = C_{10}$, and $\phi_1(e_1)=%
\phi_1(e_2)=C_{10},$ and $\phi_1(x_i)=1$ for every remaining canonical generator $y$ of $\Gamma$  
(namely $e_i$ for $3 \le i \le k$, and $x_i$ for $1 \le i \le r$, and every $c_{ij}$ with $2 \le i \le k$).
This makes $\phi_1$ an epimorphism (whether $s_1=0$ or $s_1\ge3)$,  
and its kernel $\ker\phi_1$ contains no non-trivial element 
of $\langle c_{1j},c_{1j+1} \rangle \cong {\rm D}_{n_{1j+1}}$ for any $j \ge 0$,  
nor any conjugate of such an element, 
and also $\ker\phi_1$ contains the orientation-reversing element $c_{10}e_1.$

We may also define $\phi_2:\Gamma\to G_{(n_{21},\ldots,n_{2s_2})}$ by setting $%
\phi_2(c_{2j})=C_{2j}$ for $0\le j\le s_2-1,$ and $\phi(c_{2s_2}) = C_{20}$, 
and $\phi_2(z)=1$ for every remaining canonical 
generator $z$ of $\Gamma$ (namely $e_i$ for $1 \le i \le k$, and $x_i$ for $1 \le i \le r$, 
and every $c_{ij}$ with $i \ne 2$).
This makes $\phi_2$ a homomorphism (whether $s_2=0$ or $s_2\ge3)$, 
and its kernel $\ker\phi_2$ contains no non-trivial element 
of $\langle c_{2j},c_{2j+1} \rangle \cong {\rm D}_{n_{2j+1}}$ for any $j \ge 0$, 
nor any conjugate of such an element, 
and again $\ker\phi_2$ contains the orientation-reversing element $c_{10}e_1.$

For the remaining period cycles $\mathcal{C}_i$ for $3 \le i \le k$ (if any), we may define the
following homomorphisms according to the number $s_i$ of link periods of $\mathcal{C}_i.$

If $\mathcal{C}_i=(n_{i1},\ldots,n_{is_i})$ with $s_i\ge3$ or $s_i=0$, 
then we may take $G = G_{(n_{i1},\ldots,n_{is_i})}$ and define a homomorphism $\phi_i:\Gamma\to G$ 
in exactly the same way as we defined $\phi_2$ above, so that $\ker\phi_i$ contains no non-trivial element 
of $\langle c_{ij},c_{ij+1} \rangle \cong {\rm D}_{n_{ij+1}}$ for any $j \ge 0$, 
nor any conjugate of such an element, 
and again $\ker\phi_i$ contains $c_{10}e_1.$

If $\mathcal{C}_i=(n_{i1},n_{i2})$, 
we may take $G = \mathrm{D}_{2n_{1}n_{2}}=\langle\, u,v \ | \ u^2=v^2=(uv)^{2n_{1}n_{2}}=1\,\rangle$ 
and define a homomorphism $\psi_i:\Gamma\to G$ 
such that $\psi_i(e_i)=(uv)^{n_2-n_1}$ and $\psi_i(e_2) = (uv)^{n_1-n_2}$, 
where $e_2$ is the hyperbolic generator associated with the period cycle $\mathcal{C}_2,$
and $\psi_i(c_{i0})=u(uv)^{2n_2}$, $\,\psi_i(c_{i1})=u$ and $\,\psi_i(c_{i2})=u(uv)^{2n_1}$, 
and $\psi(z)=1$ for every remaining canonical generator $z$ of $\Gamma.$
Then $\ker\psi_i$ contains no non-trivial element 
of $\langle c_{ij},c_{ij+1} \rangle \cong {\rm D}_{n_{ij+1}}$ for any $j \ge 0$, 
nor any conjugate of such an element, 
and again $\ker\psi_i$ contains $c_{10}e_1.$

If $\mathcal{C}_i=(n_{i1})$, 
we may take $G = \mathrm{D}_{2n_{i1}}=\langle\, u,v \ | \  u^2=v^2=(uv)^{2n_{i1}}=1 \,\rangle$ 
and define a homomorphism $\tau_i:\Gamma\to G$ 
such that  $\tau_i(e_i)=v$ and $\tau_i(e_2) = v$, and $\tau_i(c_{i0})=u$ and $\tau_i(c_{i1})=vuv$, 
and $\tau_i(z)=1$ for every remaining generator $z$ of $\Gamma.$ 
Then $\ker\tau_i$ contains no non-trivial element of $\langle c_{i0},c_{i1} \rangle \cong {\rm D}_{n_{i1}}$, nor any conjugate of such an element, 
and again $\ker\tau_i$ contains $c_{10}e_1.$

\medskip

Suppose now that the number $r$ of proper periods of $\sigma(\Gamma)$ is greater than $2$. 
Then there exists a finite group $G_{[m_1,\ldots,m_r]}$ with partial presentation
\begin{equation*}
G_{[m_1,\ldots,m_r]} = \langle X_1,\ldots, X_r\mid X_1^{m_1}=\cdots
=X_r^{m_r}=X_1\cdots X_r=\cdots =1\rangle,
\end{equation*} 
such that $X_i$ has order $m_i$ for $1 \le i \le r$; see \eqref{G[m_1,...,m_r]}. 
In this case we define a homomorphism $\xi:\Gamma\to G_{[m_1,\ldots,m_r]}$
such that $\xi(x_i)=X_i$ for $i=1,\ldots,r$ and $\xi(w)=1$ for every other canonical 
generator $w$ of $\Gamma$ (including every $a_i$ and $b_i$, every $e_i$, and every $c_{ij}$). 
Then $\ker\xi$ contains no non-trivial element of  $\langle x_i \rangle \cong {\rm C}_{m_i}$,  
nor any conjugate of such an element, for $1 \le i \le r$, 
and $\ker\xi$ contains the orientation-reversing element $c_{10}e_1.$

\medskip

If $r=2$, then more care is needed. 
Let $G_{[m_1,m_2,2]}$ be a finite smooth quotient of the ordinary $(2,m_1,m_2)$ triangle group, 
generated by three elements $X_1,$ $X_2$ and $X_3$ of orders $m_1,$ $m_2$ and $2$ respectively, 
such that $X_1X_2X_3=1.$ 
Now take $G$ as the direct product $G_{[m_1,m_2,2]}\times G_{(n_{11},\ldots,n_{1s_1})},$ 
where $G_{(n_{11},\ldots,n_{1s_1})}$ is as used previously,  
and define the homomorphism $\xi:\Gamma\to G$ by setting $\xi(x_1)=X_1$ and $\xi(x_2)=X_2,$ 
and $\xi(c_{1j})=C_{1j}$ for $0 \le j < s_1$, and $\xi(e_1)=C_{10}$ and $\xi(e_2)=C_{10}X_3$, 
where $e_2$ is the hyperbolic generator associated with the period cycle $\mathcal{C}_2,$ 
and $\xi(w)=1$ for every other canonical generator $w$ of $\Gamma.$
Then $\ker\xi$ contains no non-trivial element 
of  $\langle x_i \rangle \cong {\rm C}_{m_i}$ for $i \in \{1,2\}$,
or $\langle c_{1j} \rangle \cong {\rm C}_2$ 
or $\langle c_{1j},c_{1j+1} \rangle \cong {\rm D}_{n_{1j+1}}$ for any $j \ge 0$, 
nor any conjugate of such an element, 
and again $\ker\xi$ contains $c_{10}e_1.$

\medskip

Similarly, if $r=1$, then let $G_{[m_1,2,2]}$ be 
$\mathrm{D}_{m_1}=\langle\, u,v \ | \  u^2=v^2=(uv)^{m_1}=1 \,\rangle$, 
and take $(X_1,X_2,X_3) = (uv,v,u)$, with orders $m_1,$ $2$ and $2$ respectively, 
such that $X_1X_2X_3=1,$  
and take $G$ as the direct product 
$G_{[m_1,2,2]}\times G_{(n_{11},\ldots,n_{1s_1})}$ where $G_{(n_{11},\ldots,n_{1s_1})}$.  
This time define the homomorphism $\xi:\Gamma\to G$ by setting 
$\xi(x_1)=X_1,$ and $\xi(c_{1j})=C_{1j}$ for $0 \le j < s_1$,  
and $\xi(e_1)=C_{10}$ and $\xi(e_2)=C_{10}X_2X_3$, 
and $\xi(w)=1$ for every other canonical generator $w$ of $\Gamma.$
Then $\ker\xi$ contains no non-trivial element of $\langle x_1 \rangle \cong {\rm C}_{m_1}$,
or $\langle c_{1j} \rangle \cong {\rm C}_2$ 
or $\langle c_{1j},c_{1j+1} \rangle \cong {\rm D}_{n_{1j+1}}$ for any $j \ge 0$, 
nor any conjugate of such an element, 
and again $\ker\xi$ contains $c_{10}e_1.$

\medskip

Here we note that the kernel of every one of the homomorphisms $\xi$, $\phi_i$, $\psi_i$, $\tau_i$ and $\xi$ 
defined above has finite index in $\Gamma$, is torsion-free, and contains the orientation-reversing 
element $c_{10}e_1.$

\medskip

Finally, consider the intersection of the kernels of the homomorphisms 
$\xi$, $\phi_1$, $\phi_2$ and either $\phi_i$ or $\psi_i$ or $\tau_i$, for each $i \in \{3,\dots,k\}$, 
which we will write loosely as 
\begin{equation*}
N = \Big(\bigcap_{i}\ker\phi_i\Big)\cap \Big(\bigcap_{i}\ker\psi_i\Big)\cap \Big(\bigcap_i\ker\tau_i\Big)\cap \ker\xi. 
\end{equation*}
Then $N$ is a normal subgroup of $\Gamma$ which is torsion-free, 
as it contains no conjugate of any non-trivial element of any torsion subgroup of $\Gamma$. 
Moreover, $N$ has finite index in $\Gamma,$ by Lemma \ref{intersection.finite.index}. 
Also it contains the product $c_{10}e_1,$ which is an orientation-reversing element of $\Gamma,$ 
and so this normal group satisfies the conditions in the statement of the theorem.
\end{proof}

Unfortunately we have not been able to resolve the NEC version of Fenchel's conjecture
in the remaining case where $k_0+k_3 \le 1$, so we leave open the following: 

\begin{question}
\label{question.k>1} 
If the NEC group $\Gamma$ has signature $(0;+;[m_1,...,m_r]; \{\mathcal{C}_1,...,\mathcal{C}_k\})$ 
for some $k > 0$, with $k_0+k_3 = 0$ or $1$ (where $k_0$ is the number of empty period cycles $(-),$ 
and $k_{3}$ is the number of period cycles with three or more link periods), 
then does $\Gamma$ have a torsion-free normal subgroup of finite index 
that contains orientation-reversing elements?
\end{question}

The case {$k=1$}, with $\sigma(\Gamma)=(0;+;[m_1,\ldots,m_r];\{(n_1,%
\ldots,n_s)\})$, is perhaps the most challenging. 
For this one, we have solved some particular subcases, as listed in the theorem below.
Recall here that if $\Gamma$ is a proper NEC group, then 
the expression $\mu(\Gamma)$ given in \eqref{area} is positive.	

\begin{theorem}
\label{theorem.non-orientable.g=0.k=1.solved.cases} Let $\Gamma $ be a
proper NEC group with signature $(0;+;[m_{1},\ldots ,m_{r}];\{(n_{1},\ldots,n_{s})\}$. 
Then $\Gamma $ has a torsion-free normal subgroup of finite index in $\Gamma $ and containing orientation-reversing elements in
the cases listed in Table $\ref{table.g=0.k=1.solved.cases}$.
\end{theorem}

${}$\\[-24pt]

\begin{table}[hhh]
\renewcommand{\arraystretch}{1.5}
\par
\begin{center}
\begin{tabular}{l|l}
\multicolumn{1}{c|}{\quad $(n_1,\ldots,n_s)$} & \multicolumn{1}{c}{$[m_1,\ldots,m_r]$} 
\\ \hline\hline
$s=0$ \quad $(-)$ & \ Any 
\\ \hline
$s=1$ \quad $(n)$, \ for odd $n > 2$ \qquad & \ $[2,\ldots,2]$ 
\\ \hline
$s=2$ \quad $(2,2)$ & \ Any 
\\ \hline
$s=2$ \quad $(n,n)$ & \ $r\ge3$, \ with at least one $m_i$ even \quad 
\\ \hline
$s\ge3$ \quad $(2,2,n_3,\ldots,n_s)$ & \ $r\ge2$ 
\\ \hline
$s\ge4$ \quad $(2,2,n_3,\ldots,n_s)$ & \ $r=0$ 
\\ \hline
$s=3$ \quad $(2,2,n_3)$, with $n_3$ even if $m_1$ is odd \ & \ $r=1$, \ with $m_1$ even if $n_3$ is odd 
\\ \hline
$s\ge4$ \quad $(2,2,n_3,\ldots,n_s)$  & \ $r=1$ 
\\ \hline
$s\ge3$ \quad any & \ $r\ge2$, \ with at least one $m_i$ even \quad 
\\ \hline
$s\ge3$ \quad any & \ $r=1,$ \ with $m_1\equiv 2$ mod $4$ \quad 
\\ \hline
\end{tabular}%
${}$\\[-30pt]
\end{center}
\caption{Some cases with $g=0$ and $k=1$ in which the NEC form of Fenchel's conjecture holds}
\label{table.g=0.k=1.solved.cases}
\end{table}

\begin{proof}
For each of the ten cases listed in Table~\ref{table.g=0.k=1.solved.cases}, 
we define a homomorphism $\phi $ from the corresponding NEC group $\Gamma $ to some finite group $G$, 
with the property that $\ker \phi $ is torsion-free and contains orientation-reversing elements. 

\bigskip

\noindent {\bf Case 1}: \ 
$\sigma (\Gamma)=(0;+;[m_{1},\ldots ,m_{r}];\{(-)\})$ 

\smallskip

As $s=0$, we need $r>2$ or $r=2$ with $[m_{1},m_{2}]\neq [2,2]$ in order for  $\mu(\Gamma)$ to be positive, 
and in both cases, $r+1\geq 3$.
Now let $G$ be a finite group $G_{[m_{1},\ldots ,m_{r},2]}$ with partial presentation
\begin{equation*}
G_{[m_{1},\ldots ,m_{r},2]}=\langle X_{1},\ldots ,X_{r},X_{r+1}\mid
X_{1}^{m_{1}}=\cdots =X_{r}^{m_{r}}=X_{r+1}^{\,2}=X_{1}\cdots
X_{r}X_{r+1}=\cdots =1\rangle ,
\end{equation*}%
as in \eqref{G[m_1,...,m_r]}, 
and define the homomorphism $\phi :\Gamma \rightarrow G$ by setting 
$\phi(x_{i})=X_{i} $ for $1\leq i\leq r$ and $\phi (e_1)=\phi (c_{10})=X_{r+1}$, 
so that $\ker\phi$ contains the orientation-reversing element $e_1c_{10}$.

\bigskip

\noindent {\bf Case 2}: \ 
$\sigma (\Gamma)=(0;+;[2,\ldots ,2];\{(n)\})$ for odd $n > 2$

\medskip

Let $G = D_{n}=\langle\, u,v \ | \ u^2 = v^2 = (uv)^n = 1 \,\rangle$, 
and define the homomorphism $\phi :\Gamma \rightarrow G$ by setting 
 $\phi (c_{10})=u$ and $\phi(c_{11})=v,$ and $\phi (e_1)=(uv)^{(n+1)/2}$,  
and $\phi(x_i) =u(uv)^{\alpha _{i}}$ for $1 \le i \le r$, where the integers $\alpha_i$ 
are chosen so that that the long relation $x_{1}x_{2}\dots x_{r}e_1 = 1$ is preserved.
Then $\ker\phi$ contains the orientation-reversing element $x_{1}c_{0}(c_{0}c_{1})^{\alpha _{1}}$.

\bigskip

\noindent {\bf Case 3}: \ 
$\sigma (\Gamma)=(0;+;[m_{1},\ldots ,m_{r}];\{(2,2)\})$  

\medskip

When $r=1$, let $G$ be the direct product $H \times {\rm C}_2$, 
where $H$ is a finite smooth quotient of the ordinary $(2,m_1,3)$ triangle group, 
generated by elements $X_{1},$ $X_{2}$ and $X_{3}$ of orders $2, m_1$ and $3$ 
such that $X_{1}X_{2}X_{3}=1$, and ${\rm C}_{2}=\langle u \rangle$, 
and define the homomorphism $\phi :\Gamma \rightarrow G$ by setting 
$\phi (x_1)=X_{2}$ and $\phi (e_1)=X_{2}^{-1}$, 
and $\phi (c_{10})=X_{1}$ and $\phi(c_{11})=u$ 
(and also $\phi(c_{12}) = \phi(e_{1}c_{10}e_{1}^{-1}) = (X_{3}X_{2})^{-1})$, 
so that $\phi(c_{12}c_{10})$ has order $2$). 
In this case $\phi(c_{0}e_{1}^{-1}) = X_{1}X_{2} = X_{3}^{-1}$, 
and so $\ker\phi$ contains the orientation-reversing element  $(c_{10}e_{1}^{-1})^{3}.$

When $r\geq 2$, let $G$ be the direct product $H \times {\rm C}_2$, 
where $H = G_{[m_{1},\ldots ,m_{r},2]}$ is the finite group with partial presentation as in 
Case 1 above, 
and again ${\rm C}_{2}=\langle u \rangle$,  
and define the homomorphism $\phi :\Gamma \rightarrow G$ by setting 
$\phi (x_{i})=X_{i}$ for $1 \le i \le r$, and  $\phi(e_1)=X_{r+1}$, 
and $\phi (c_{10})=u$ and $\phi (c_{11})=X_{r+1}$. 
Then $\ker\phi$ contains the orientation-reversing element $e_{1}c_{11},$ 
since $X_{r+1}^{\, 2} = 1$. 

\bigskip

\noindent {\bf Case 4}: \ 
$\sigma (\Gamma )=(0;+;[m_{1},\ldots ,m_{r}];$ $\{(n,n)\})$ for $r\geq 3$,  with at least one $m_{i}$ even

\smallskip

In this case, up to an automorphism of $\Gamma $ we may assume that $m_{r}$ is even. 
Now take $G$ as the direct product $H \times \mathrm{D}_{n}$, 
 where $H$ is a finite group $G_{[m_{1},\ldots ,m_{r}]}$ with partial presentation:
\begin{equation*}
G=\langle\, X_{1},\ldots ,X_{r} \ | \ X_{1}^{m_{1}}=\cdots
=X_{r}^{m_{r}}=X_{1}\cdots X_{r}=\cdots =1 \,\rangle,
\end{equation*}
as in \eqref{G[m_1,...,m_r]}, 
and $\mathrm{D}_{n}=\langle\, u, v \ | \ u^2 = v^2 = (uv)^{n}=1 \,\rangle $. 
and define the homomorphism $\phi :\Gamma \rightarrow G$ by setting 
$\phi (x_{j})=X_{j}$ for $1\leq j \le r-1$ and $\phi (x_{r})=X_{r}u$, 
and $\phi (e_1)=u,$ and $\phi(c_{10})=u$, $\phi (c_{11})=v$ and $\phi (c_{12})=u.$
Then $\ker\phi$ contains the orientation-reversing element $e_{1}c_{10}.$ 

\bigskip

\noindent {\bf Case 5}: \ 
$\sigma (\Gamma)=(0;+;[m_{1},\ldots ,m_{r}];$ $\{(2,2,n_{3},\ldots,n_{s})\})$ for $r\geq 2$ and $s\geq 3$ 

\medskip

Here we could have assumed that the single period cycle is $(n_{1},\ldots ,n_{i-1},2,2,n_{i+2},\ldots,n_{s})$, 
with two successive $2$s, but up to an automorphism of the NEC group $\Gamma$, we may 
take it with the given form. 

In this case, let $H$ be a finite group $G_{[m_{1},\ldots ,m_{r},2]}$ with partial presentation as in 
Case 1 above,
and let $J$ be a  finite group $G_{(2,2,n_{3},\ldots ,n_{s})}$ with partial presentation
\begin{equation}  \label{G(n_1,...,n_s)}
G_{(n_1,...,n_s)} \!= \!\langle\, C_0,..., C_{s-1} \ | \ C_i^{\,2}\!=\!
(C_0C_1)^{n_1} \!= \! \cdot\cdot\cdot \!=\!(C_{s-2}C_{s-1})^{n_{s-1}} \!=\!
(C_{s-1}C_0)^{n_s} = \cdot\cdot\cdot =1\,\rangle.
\end{equation}
as in \eqref{G(n_1,...,n_s)}, with $(n_1,n_2) = (2,2)$, 
and take $G$ as the direct product $H\times J$. 
Now define the homomorphism $\phi :\Gamma \rightarrow G$ by setting 
$\phi (x_{j})=X_{j}$ for $1\leq j\leq r,$ and  $\phi(e_1)=X_{r+1}$, 
and $\phi (c_{10})=C_{0},$ \ $\phi (c_{11})=X_{r+1}$, and $\phi(c_{1i})=C_{i}$ for $2\leq i\leq s-1,$  
and $\phi(c_{1s}) = C_0$. 
Then $\ker\phi$ contains the orientation-reversing element $e_{1}c_{11}.$

\bigskip

\noindent {\bf Case 6}: \ 
$\sigma (\Gamma)=(0;+;[-];\{(2,2,n_{3},\ldots ,n_{s})\})$ for $s\geq 4$ 

\medskip

As above, we could have assumed that the period cycle is $(n_{1},\ldots ,n_{i-1},2,2,n_{i+2},\ldots,n_{s})$, 
but we may take it with the given form.  
Note also that $s \ge 4$, in order to ensure that $\mu(\Gamma) > 0$. 

Here we can take $G$ as a finite group $G_{(2,n_{3},\ldots ,n_{s})}$ 
generated by $s-1\geq 3$ involutions with partial presentation
\begin{equation*}
\langle\, C_{1},\ldots ,C_{s-1} \ | \ 
C_{i}^{\,2}=(C_{1}C_{2})^{2}=(C_{2}C_{3})^{n_{3}}=\cdots
=(C_{s-2}C_{s-1})^{n_{s-1}}=(C_{s-1}C_{1})^{n_{s}}=\cdots =1\rangle,
\end{equation*}%
and define the homomorphism $\phi :\Gamma \rightarrow G$ by setting 
$\phi(e_1) = 1,\,$ and $\phi(c_{10})=\phi(c_{1s})=C_{1}$ and $\phi(c_{11})=C_{1}C_{2}$,  
and $\phi (c_{1i})=C_{i}$ for $2\leq i\leq s-1.$
In this case $\ker\phi$ contains the orientation-reversing element  $c_{10}c_{11}c_{12}.$
 
\bigskip

\noindent {\bf Case 7}: \ 
$\sigma (\Gamma)=(0;+;[m];\{(2,2,n_{3})\})$ where $m$ or $n_3$ is even

\medskip

Let $G$ be the direct product of ${\rm D}_{n_3}=\langle\, u,v \ | \ u^2 = v^2 = (uv)^{n_3} = 1 \,\rangle$ 
and ${\rm C}_m = \langle\, w \ | \ w^m = 1 \,\rangle$.
If $m$ is even,  then define the homomorphism $\phi :\Gamma \rightarrow G$ by setting $\phi (x_1)=w$, 
$\phi (e_1)=w^{-1}$,  and $\phi (c_{10})=\phi(c_{13})=u,\,$ $\phi (c_{11})=w^{m/2}$ and $\phi (c_{12})=v$. 
In this case $\ker\phi$ contains the orientation-reversing element $e_1^{\,m/2}c_{11}.$
On the other hand, if $n_3$ is even, define the homomorphism $\phi :\Gamma \rightarrow G$ 
by setting $\phi (x_1)=w$, $\phi (e_1)=w^{-1}$,  and $\phi (c_{10})=\phi(c_{13})=u,\,$ 
$\phi (c_{11})=(uv)^{n_3/2}$ and $\phi (c_{12})=v$, 
and then $\ker\phi$ contains the orientation-reversing element $(c_{10}c_{12})^{\,n_3/2}c_{11}.$

\bigskip

\noindent {\bf Case 8}: \ 
$\sigma (\Gamma)=(0;+;[m];\{(2,2,n_{3},\ldots ,n_{s})\})$ for $s\geq 4$ 

\medskip

As above, we could have assumed that the period cycle is $(n_{1},\ldots ,n_{i-1},2,2,n_{i+2},\ldots,n_{s})$, 
but we may take it with the given form.  

Let $G$ be the direct product $H \times {\rm C}_m$ of a finite group $H = G_{(2,n_{3},\ldots ,n_{s})}$ 
with partial presentation  as in 
Case 6 above, 
with ${\rm C}_m = \langle\, w \ | \ w^m = 1 \,\rangle$,   
and define the homomorphism $\phi :\Gamma \rightarrow G$ by setting 
$\phi (x_1)=w$  and $\phi(e_1)=w^{-1}$,  and $\phi (c_{10})=\phi(c_{1s}) =C_{1}$ and $\phi(c_{1i})=C_{1}C_{2},$ 
and $\phi (c_{1i})=C_{i}$ for $2\leq i\leq s-1$. 
Then $\ker\phi$ contains the orientation-reversing element $c_{10}c_{11}c_{12}.$

\bigskip

\noindent {\bf Case 9}: \ 
$\sigma (\Gamma )=(0;+;[m_{1},\ldots ,m_{r}];\{(n_{1},\ldots ,n_{s})\})$ with at
least one $m_{i}$ even, $r\geq 2$ and $s\geq 3$ 

\medskip

In this case, without loss of generality we may assume that $m_{r}$ is even.
(Also note that the case $s=0$ was considered in Case 1 above, for any 
choice of $[m_{1},\ldots ,m_{r}]$.) 

Let $G = H \times J$, 
where $H$ is a finite group $G_{[m_{1},\ldots ,m_{r},2]}$ with partial presentation 
as in 
Case~1 above, 
and $J$ is a finite group $G_{(n_{1},\ldots ,n_{s})}$ with partial presentation
\begin{equation*}
J=\langle\, C_{0},\ldots ,C_{s-1} \ | \ C_{i}^{\,2}
=(C_{0}C_{1})^{n_{1}}=\cdots
=(C_{s-2}C_{s-1})^{n_{s-1}}=(C_{s-1}C_{0})^{n_{s}}=\cdots =1 \,\rangle,  
\end{equation*}%
as  in \eqref{G(n_1,...,n_s)}.
Now define the homomorphism $\phi :\Gamma \rightarrow G$ by setting 
$\phi (x_{j})=X_{j}$ for $1\leq j\leq r-1,$ and $\phi (x_{r})=X_{r}C_{0},$ 
and $\phi (e_1)=X_{r+1}C_{0}$, 
and $\phi (c_{1i})=X_{r+1}C_{i}$ for $0\leq i\leq s-1,$ and also $\phi(c_{1s}) = X_{r+1}C_0$.  
Then $\ker\phi$ contains the orientation-reversing element $e_{1}c_{10}.$

\bigskip

\noindent {\bf Case 10}: \ 
 $\sigma (\Gamma)=(0;+;[m];\{(n_{1},\ldots ,n_{s})\})$ with $m\equiv 2$ mod $4$ and $s\geq 3$

\medskip

Let $G = {\rm C}_m \times J$, 
where ${\rm C}_{m}=\langle\, w \ | \ w^{m}=1 \,\rangle$ 
and $J$ is a finite group $G_{(n_{1},\ldots ,n_{s})}$ with partial presentation as in 
Case 9 above, 
and define the homomorphism $\phi :\Gamma \rightarrow G$ by setting 
$\phi (x_1)=wC_0$, $\phi (e_1)=(wC_0)^{-1}$, 
and $\phi (c_{1i})=w^{m/2}C_{i}$ for $0\leq i < s$, and $\phi (c{_1s})=w^{m/2}C_0$. 
Then since $m/2$ is odd, $\ker\phi$ contains the orientation-reversing element $e^{m/2}c_{10}.$
\end{proof}

In order to get a clearer overview of the results obtained so far, we add to
Table \ref{table.g=0.k=1.solved.cases} the unresolved cases 
for signature $(0;+;[m_{1},\ldots ,m_{r}];\{(n_{1},\ldots,n_{s})\}$, 
in Table \ref{table.g=0.k=1.unresolved.cases}.
\begin{table}[]
\renewcommand{\arraystretch}{1.5}
\par
\begin{center}
\begin{tabular}{l|l}
\multicolumn{1}{c|}{\qquad $(n_1,\ldots,n_s)$} 
& \multicolumn{1}{c}{$[m_1,\ldots,m_r]$} 
\\ \hline\hline
$s=1$ \quad $(n)$ & \ Any, except for $[2,\ldots,2]$ when $n$ is odd \ 
\\ \hline
$s=2$ \quad $n_1 \ne n_2$ & \ Any 
\\ \hline
$s=2$ \quad $(n,n)$ where $n > 2$ \quad & \ $r = 1$ or $2$, or $r \ge 3$ with all $m_i$ odd 
\\ \hline
$s=3$ \quad Any with all $n_i > 2$ \ \ & \ $r=0$  \ 
\\ \hline
$s=3$ \quad $(2,2,n_3)$ for odd $n_3$ \ \ & \ $r=1$ with  $\,m_1$ odd \ 
\\ \hline
$s\ge3$ \quad Any with $(n_1,n_2) \ne (2,2)$ \ \ & \ Any (with $r \ge 0$)   
\\ \hline
$s\ge3$ \quad Any & \ $r=1,$ with $\,m_1\not\equiv2$ mod $4$ \ 
\\ \hline
$s\ge3$ \quad Any & \ $r\ge2$, with all $m_i$ odd 
\\ \hline
\end{tabular}
\end{center}
\caption{Summary of unresolved cases when $g=0$ and $k=1$}
\label{table.g=0.k=1.unresolved.cases}
\end{table}

\medskip

Note that there are some natural restrictions on the signatures occurring in Table \ref{table.g=0.k=1.unresolved.cases}.
In the case of signature $(0;+;[m_{1},\ldots ,m_{r}];\{(n)\})$, for example, 
 if $r=1$ then we have to exclude signatures $(0;+;[2];\{(n)\})$ for every $n,$
and $(0;+;[m];\{(n)\})$ for $(m,n)=(3,2)$, $(3,3)$ and $(4,2)$, 
because these are signatures corresponding to spherical or Euclidean crystallographic groups, 
which do not have fundamental regions with positive area.

\newpage 
\section{Some special cases}
\label{Special cases}

\subsection{Comments about signature \protect\boldmath$\protect\sigma(\Gamma)=(0;-;[-];\{(n_1,\ldots,n_s)\})$} 

The special case of the signature $(0;+;[-]; \{(n_1,\ldots,n_s)\})$ is quite a challenging one, 
but also turns out to be quite interesting.
Observe that an NEC group $\Gamma$ with such a signature is generated by reflections, 
say $c_0,\ldots,c_{s-1},$ and then any orientation-reversing element of $\Gamma$ is a word of odd
length in those generators, and so we obtain the following:

\begin{proposition}
\label{(n1,...,ns)} An NEC group $\Gamma$ with signature $%
(0;+;[-];\{(n_1,\ldots,n_s)\})$ has a torsion-free normal subgroup with 
finite index containing orientation-reversing elements if
and only if there exists a finite group $G=G_{(n_1,\ldots,n_s)}$ 
generated by $s$ involutions $C_0,\ldots, C_{s-1}$ such that 
$C_{i-1}C_{i}$ has order $n_i$ for $1 \le i <s$ and $C_{s-1}C_0$ has order $n_s$, 
and the identity element of $G$ is expressible as a word $w$ of odd length in the generators $C_0,C_1,\ldots,C_{s-1}$.
\end{proposition}

\begin{proof}
It is easy to see that if $\Gamma$ contains such a torsion-free normal subgroup $K$, then $%
\Gamma/K$ is a finite group $G$ with a generating set $\{C_0,C_1,\dots,C_{s-1}\}$ 
of the given form,
Conversely, if there exists such a finite group $G$, 
then the assignments $e_1 \mapsto 1$ and $c_i \mapsto C_i$ for $0\le i < s$ 
extend to an epimorphism $\theta: \Gamma \to G$ whose kernel $K$ has index $|G|$ 
and is torsion-free. 
Also if the identity element of $G$ is expressible as a word $w$ of odd length as given, 
then the corresponding word in the generators $c_0,c_1,\ldots,c_{s-1}$ of $\Gamma$ is an 
orientation-reversing element that lies in $K$.
\end{proof}

Following Singerman's study of H$^*$-groups in \cite{singerman:non-orientable},
we have the following alternative characterisation of the groups $G$ 
that satisfy the condition given in Theorem~\ref{(n1,...,ns)}. 

\begin{corollary}
\label{[n1,...,ns]} 
An NEC group $\Gamma$ with signature $(0;+;[-];\{(n_1,\ldots,n_s)\})$ has a torsion-free 
normal subgroup of finite index containing orientation-reversing elements if
and only if there exists a finite group $G=G_{[n_1,\ldots,n_s]}$ generated
by $s$ elements $X_1,\ldots, X_s$ of respective orders $n_1,\ldots,n_s$ 
such that $X_1\cdots X_s=1$ and that $G$ contains an element $Z$ such that
the element $ZX_1\cdots X_i$ is an involution for $0 \le i < s$. 
\end{corollary}

\begin{proof}
Suppose $\Gamma$ contains such a normal subgroup, say $K$.  
Then $\Gamma/K$ is a finite group $G$ generated by $s$ involutions $C_0,\ldots,C_{s-1}$ 
such that the element $X_i:=C_{i-1}C_i$ has order $n_i$ for $1\le i\le s-1$ 
and $X_s:=C_{s-1}C_0$ has order $n_s,$ and $X_1\cdots X_s = 1$. 
Moreover, since $K$ contains an orientation-reversing element, there exists a word 
$w=w(c_{10},\ldots,c_{1{s-1}})$ of odd length in the generators $c_{10},\ldots,c_{1{s-1}}$ 
of $\Gamma$ that lies in $K$. (Note here that the only other canonical generator $e_1$ is necessarily trivial.) 
Next, if $c_{1{i_0}}$ is the first letter of this this word $w$, then because the image of $w$ in $G$ is trivial, 
the image $C_{1{i_0}}$ of $c_{1{i_0}}$ can be expressed in terms of the pairwise products 
$C_{1i}C_{1j}$, and hence in terms of the elements $X_1,\dots,X_s$, 
and it follows that $G$ is generated by $\{X_1,\ldots,X_s\}$.  
Now we can take $Z=C_0$, and then an easy induction shows that 
$C_i=ZX_1\cdots X_i$ for $1\le i < s$, and so the conclusions about the group $G$ all hold.

Conversely, suppose that some finite group $G=G_{[n_1,\ldots,n_s]}$ satisfies 
the conditions given in the statement of the theorem. 
Then we may define a homomorphism $\theta:\Gamma\to G$ by setting 
$\theta(c_{10})=Z$ \ and \ $\theta(c_{1i})=ZX_1\cdots X_i$ for $1\le i < s$ 
(and $\theta(e_1) = 1$). 
In particular, the image $\langle C_{j-1},C_j \rangle$ of $\langle c_{1{j-1}},c_{1j} \rangle$ 
is dihedral of order $2n_j$, for $1 \le j < s$, and similarly, the image of $\langle C_{s-1},C_0 \rangle$ 
of $\langle c_{1{s-1}},c_{10} \rangle$ is dihedral of order $2n_s$, 
and hence the finite index normal subgroup $\ker\theta$ of $\Gamma$ is torsion-free.
Finally, the involution $Z$ is expressible as a word $w(X_1,\ldots,X_{s-1})$ in the 
generators $X_1,\ldots,X_{s-1}$ of $G$, so that $Z w(X_1,\ldots,X_{s-1}) = 1$ in $G$, 
and then since $\theta(c_0w(c_0c_1,\ldots,c_{s-2}c_{s-1})) = Z w(X_1,\ldots,X_{s-1})$,
it follows that $c_0w(c_0c_1,\ldots,c_{s-2}c_{s-1})$ is an orientation-reversing element 
lying in $\ker\theta.$
\end{proof}

For the rest of this section, we explain how there exist very many finite groups $G$ that have 
the property stated in Proposition \ref{(n1,...,ns)}. 

\subsection{Examples coming from non-orientable regular maps and hypermaps} 
\label{nonomaps}

First, consider the signature $(0;+;[-];\{(2,k,m)\})$, with $s = 3$. 
Every NEC group $\Gamma$ with such a signature is isomorphic to the ordinary $(2,k,m)$ 
triangle group, and hence its smooth homomorphic images occur as automorphism groups 
of regular maps of type $\{m,k\}$, as described in \cite{coxeter-moser} for example. 
Such a map $M$ lies on an orientable surface if and only if the image of the 
canonical Fuchsian subgroup $\Gamma^+$ has index $2$ in $\aut(M)$, in which case 
that image is the group $\aut^+(M)$ of all orientation-preserving automorphisms of $M$. 
On the other hand, $M$ lies on a non-orientable surface if and only if that image coincides with $\aut(M)$, 
which happens if and only if the kernel of the homomorphism from $\Gamma$ to $\aut(M)$ 
contains orientation-reversing elements. 
Accordingly, the family of non-orientable regular maps of given type $\{m,k\}$ 
provides us with torsion-free subgroups of finite index in $\Gamma$ giving a positive 
answer to the NEC group version of Fenchel's conjecture. 

Indeed there are well over 3000 such maps (with various types) on non-orientable surfaces 
of genus $4$ to $602$ (see \cite{conder.nonomaps}),   
and a way of constructing examples with specific types (and infinite families of types 
with one of $k$ and $m$ fixed and the other variable) was given in 1995 in \cite{conder-everitt}.  
 
Much more recently, it has been shown that there exists a regular map of type $\{m,k\}$ 
on a non-orientable surface of genus $p > 2$ for every pair of integers $k$ and $m$ such that 
$\frac{1}{k} + \frac{1}{m} < \frac{1}{2}$ (see \cite{asciak-conder-reade-siran}).
Hence we have a positive answer to the NEC group version of Fenchel's conjecture 
for all `hyperbolic' signatures $(0;+;[-];\{(2,k,m)\})$ for which  $\frac{1}{k} + \frac{1}{m} < \frac{1}{2}$.

(Incidentally, there are also such maps on the real projective plane, but of course in those cases 
the area of a fundamental region for $\Gamma$ would not be positive, so we can ignore them.)

\medskip

An obvious open question is whether or not this extends to all signatures 
$(0;+;[-];\{(k,l,m)\})$ for which  $\frac{1}{k} + \frac{1}{l} + \frac{1}{m} < 1$, 
via consideration of regular `hypermaps'. 
All examples of such hypermaps with genus $2$ to $202$ (up to isomorphism and triality) 
are given at \cite{conder.nonohypermaps}.

\medskip

It is also easy to construct examples with signature $(0;+;[-];\{(n,\overset{s}{\ldots },n)\})$ 
from non-orientable regular maps of type $\{s,2n\}$, for certain values of $n$ and $s$. 

\medskip

Take the NEC group $\Gamma$ with signature $(0;+;\{(2,s,2n)\})$, 
also known as the full $(2,s,2n)$ triangle group, with presentation 
$\langle\, c_0,c_1,c_2 \ | \ c_0^{\,2} = c_1^{\,2} = c_2^{\,2} = (c_0 c_1)^{\,2} = (c_1 c_2)^{\,s} 
= (c_2 c_0)^{\,2n} = 1 \, \rangle$, 
and let $G$ be the automorphism group of a non-orientable map $M$ of type $\{s,2n\}$.  
Then there exists an epimorphism $\phi :\Gamma \rightarrow G$, with torsion-free kernel $K = \ker\phi$,  
and by non-orientability, we know that $K$ contains orientation-reversing elements.  
Moreover, if $C_i = \phi(c_i)$ for $1\le i \le 3$, then $G$ is generated 
by $C_{0}C_{1}$ and $C_{1}C_{2}$, for the same reason.

Now let $N$ be the subgroup of $G$ generated by the involutions 
$C_j^{\,\prime} = (C_{1}C_{2})^{-j}C_{0}(C_{1}C_{2})^{j}$ for $0 \le j < s$. 
This is normalised by $C_{1}C_{2}$, conjugation by which permutes its generators, 
and is also normalised by $C_{0}C_{1} = C_{1}C_{0}$, because 
$$
C_{0}C_{1} (C_{1}C_{2})^{-j}C_{0}(C_{1}C_{2})^{j} C_{1}C_{0} 
= C_{0} (C_{1}C_{2})^{j}C_{1}C_{0}C_{1}(C_{1}C_{2})^{-j} C_{0} 
= C_{0} (C_{1}C_{2})^{j} C_{0} (C_{1}C_{2})^{-j} C_{0},
$$
which lies in $N$ (for $0 \le j < s$).  
Hence $N$ is normalised by $\langle C_{0}C_{1}, C_{1}C_{2} \rangle = \langle C_{0},C_{1},C_{2} \rangle = G$.  
Furthermore, 
$C_0^{\,\prime} C_1^{\,\prime} = C_0 \, (C_{2}C_{1})C_{0}(C_{1}C_{2})= (C_0 C_2)^2$, 
which has order $n$, and it follows that  
$C_j^{\,\prime} C_{j+1}^{\,\prime} 
= (C_{1}C_{2})^{-j}C_0^{\,\prime} C_1^{\,\prime}(C_{1}C_{2})^{j}$ 
has order $n$ for $1 \le j < s$ as well. 

The quotient $G/N$ is generated by the cosets $NC_1$ and $NC_2$, and hence is a factor group 
of the dihedral group ${\rm D}_s$ of order $2s$, and if the order of $G/N$ is actually $2s$,  
then $\phi ^{-1}(N)$ is a normal subgroup $\Delta$ of index $2s$ in $\Gamma$,  
with signature $(0;+;[-];\{(n,\overset{s}{\ldots },n)\})$. 
In particular, in that case $\Delta$ is an NEC group with canonical generating set 
$\{c_{0}^{\,\prime },...,c_{s-1}^{\,\prime} \}$  
such that $\phi(c_{j}^{\,\prime}) = C_j^{\,\prime}$ 
for $0 \le j < s$, 
and then the kernel of the corresponding homomorphism $\phi^{\prime} :\Delta \rightarrow N$ is equal to $\ker\phi = K$. Thus $K$ is a torsion-free normal subgroup of finite index $|N|$ in $\Delta$, containing 
orientation-reversing elements.

\medskip

This approach works for the infinite family of non-orientable regular maps of type $\{3t,4\}$ 
given in \cite{conder-everitt}, so gives a positive answer for the NEC groups with 
signature $(0;+;[-];\{(2,\overset{s}{\ldots },2)\})$ for every $s$ divisible by $3$.  
Other such non-orientable maps of small genus give positive answers for 
signatures $(0;+;[-];\{(n,\overset{s}{\ldots },n)\})$ 
for pairs $(s,n)$ such as $(3,4)$, $(3,5)$, $(3,6)$, $(4,5)$, $(5,2)$, $(5,3)$ and $(5,4)$. 

\subsection{Examples coming from non-orientable regular polytopes} 

Another approach involves the automorphism groups of abstract regular polytopes,  
which are generalisations of both geometric polytopes and regular maps, 
as described in detail in \cite{mcmullen-schulte}). 

\medskip

If ${\cal P}$ is a regular polytope of rank $n$ with \textit{Schl{\" a}fli} type $\{k_1,\dots,k_{n-1}\}$, 
then its automorphism group is a quotient of the Coxeter group $[k_1,\dots,k_{n-1}]$,  
which has a defining presentation in terms of $n$ generators $x_1,\dots,x_n$,  
subject to the relations 
\\[+6pt] 
${}$ \ \ $x_i^{\,2} = 1\,$ for $1 \le i \le n$, \quad 
$(x_{i} x_{i+1})^{k_i} = 1\,$ for $1 \le i < n$, \ and \ 
$(x_{i} x_{j})^2 = 1\,$ for $1 \le i < j \le n$.
\\[+6pt] 
The images of the generators $x_i$ in $\aut({\cal P})$ must also satisfy certain other requirements, 
and notably something called the Intersection Condition (motivated by geometric considerations), 
and then the regular polytope ${\cal P}$ may be called \textit{orientable} or \textit{non-orientable}, 
depending on whether the image in $\aut({\cal P})$ of the index $2$ subgroup of the Coxeter group 
generated by the elements $x_{i} x_{j}$ has index $2$ or $1$. 

In particular, for any such polytope ${\cal P}$, the group $\aut({\cal P})$ is a quotient 
of the NEC group $\Gamma$ with signature $(0;+;[-];\{(k_1,k_2,\ldots,k_{n-1})\})$, 
via some (smooth) epimorphism $\theta: \Gamma \to \aut({\cal P})$ taking each 
generating reflection $c_i$ ($= c_{1i}$) of $\Gamma$ to $x_{i+1}$, for $0 \le i < n$,  
and preserving the orders of those $c_i$ and their pairwise products $c_{i}c_{i+1}$ 
(for $0 \le i \le n-2$) and $c_{n-1}c_0$. 
Moreover, the kernel of $\theta$ contains orientation-reversing elements 
if and only of the polytope ${\cal P}$ is non-orientable. 

Hence we have a positive answer to the Fenchel conjecture for a NEC group with 
signature $(0;+;[-];\{(k_1,k_2,\ldots,k_{n-1})\})$ whenever there exists a non-orientable 
regular polytope with Schl{\" a}fli type $\{k_1,\dots,k_{n-1}\}$. 

Examples are plentiful, and include not only a large number of the regular 
polytopes of ranks $3$ to $5$ listed at \cite{conder.regpolytopes}, but also an 
infinite family, consisting of a non-orientable regular polytope of rank $n$ 
with Schl{\" a}fli type $\{3,4,4\dots,4\}$ for each $n \ge 3$.  
This family starts with the examples of ranks $3$ to $5$ at \cite{conder.regpolytopes}, 
with $24$, $192$ and $768$ flags, respectively, and the remainder can be constructed 
as a `mix' (see \cite[Chapter 7]{mcmullen-schulte}) of the third one of those small examples 
with a `minimal' regular polytope of rank $n-3$ of type $\{4,\dots, 4\}$ as described in \cite{conder.smallestregpolytopes}. 

\subsection{Other examples coming from non-orientable regular maps} 

Let $\Gamma$ be an NEC group with signature $(0;+;\{(2,m,2n)\})$, 
as in Subsection~\ref{nonomaps}, but with $s$ replaced by $m$.
Then the canonical Fuchsian subgroup $\Gamma^+$ of index $2$ in $\Gamma$ 
is generated by $c_{10}c_{11}$, $c_{11}c_{12}$ and $c_{12}c_{10}$, but also $\Gamma$ 
has at least two other subgroups of index $2$, including the subgroup  
generated by $x = c_{11}c_{12}$ and $c = c_{10}$. The latter subgroup $\Lambda$ is an NEC group 
with signature $(0;+;[m];\{(n)\})$, and presentation $\langle\, x,c \ | \ x^m = c^2 = [x,c]^n = 1 \,\rangle$. 
In particular, $x$ preserves orientation, while $c$ reverses it. 
Similarly, the intersection $\Gamma^+ \cap \Lambda$ is a normal subgroup of index $4$ in $\Gamma$, 
generated by the element $x$ and its conjugate $x^c$, both of which preserve orientation, 
and have order $m$, with $x^{-1}(x^c) = [x,c]$ having order $n$.  
(In fact $\Gamma^+ \cap \Lambda$ is isomorphic to the ordinary $(m,m,n)$ triangle group.) 

Next, let $G$ be the automorphism group of a non-orientable map $M$ of type $\{m,2n\}$.  
Then there exists a homomorphism $\psi :\Gamma \rightarrow G$, such that $K = \ker\psi$ is 
torsion-free and contains orientation-reversing elements, expressible as words of odd length 
in the generators $c_{1j}$ of $\Gamma$. 

Now suppose $G = \aut(M)$ is perfect, so that $G = [G,G]$, which means that $G$ has 
trivial abelianisation.  Then $\psi(\Gamma^+ \cap \Lambda)$ cannot have index $2$ or $4$ in $G$, 
so must have index $1$ in $G$, and it follows that $G = \psi(\Lambda) = \psi(\Gamma^+ \cap \Lambda)$.
In particular, $\psi(c)$ is expressible as a word in the generators $\psi(x)$ and $\psi(x^c)$ 
of $\psi(\Gamma^+ \cap \Lambda) = G$, and so there exists an element $w \in K = \ker\psi$ 
expressible as a word in $x$ and $c$ containing an odd number of occurrences of $c$. 

Hence there exists a homomorphism $\xi : \Lambda \rightarrow G$, 
such that $\ker\xi = K \cap \Lambda$ is a torsion-free subgroup of finite index $|G$ in $\Lambda$ 
and contains an orientation-reversing element. 
In other words, the NEC form of Fenchel's conjecture holds for the group $\Lambda$ 
with signature $(0;+;[m];\{(n)\})$. 

\medskip

Of course a natural question to ask is how often this can happen. 

Well, in the proof in \cite{asciak-conder-reade-siran} that there exists a non-orientable 
regular map $M$ of (hyperbolic) type $\{m,k\}$ for every pair of integers $k$ and $m$ 
such that $\frac{1}{k} + \frac{1}{m} < \frac{1}{2}$, there are many cases to consider, 
but almost all of them involve a construction in which the group $G = \aut(M)$ is either 
$\PSL(2,q)$ for some prime-power $q > 3$, or the alternating group ${\rm A}_n$ for some $n \ge 5$. 
Actually it would be reasonable to conjecture that there always exists a non-orientable 
regular map $M$ of type $\{m,k\}$ with $\aut(M)$ isomorphic to $\PSL(2,q)$ for some such $q$, 
or isomorphic to ${\rm A}_n$ for some such $n$, or even that there exists one of each kind. 

If such a conjecture is valid, then since these groups $\PSL(2,q)$ and ${\rm A}_n$ are simple 
and hence perfect, the NEC form of Fenchel's conjecture will hold for the group $\Lambda$ 
with signature $(0;+;[m];\{(n)\})$ whenever $0 < \mu(\Lambda) = \frac{1}{2} - \frac{1}{m} - \frac{1}{2n}$, 
that is, whenever $\frac{1}{m} + \frac{1}{2n} < \frac{1}{2}$. 
But proving that lies beyond the scope of this paper. 


\end{document}